\documentclass[11pt]{amsart}
\usepackage{latexsym,amssymb,amsfonts,amsmath,amsthm,graphicx,url,color, verbatim, stmaryrd, fullpage}
\usepackage[all]{xy}
\usepackage{verbatim}
\usepackage{accents}
\usepackage[usenames, dvipsnames]{xcolor}
\usepackage{ytableau}
\usepackage{tikz}

\DeclareMathOperator*{\pro}{Pro}
\DeclareMathOperator*{\incpro}{IncPro}
\DeclareMathOperator*{\togpro}{TogPro}
\DeclareMathOperator*{\jdtpro}{JdtPro}
\DeclareMathOperator*{\jdt}{jdt}
\DeclareMathOperator*{\row}{\rm Row}
\DeclareMathOperator{\kpro}{{\it K}-Pro}
\DeclareMathOperator{\kbk}{{\it K}-BK}
\DeclareMathOperator{\rk}{\mathrm{rk}}
\newcommand{\inc}[2]{\mathrm{Inc}^{#1}(#2)}
\newcommand{\winc}[2]{\mathrm{Inc}^{#1}_{\leq}(#2)}

\newtheorem{theorem}{Theorem}[section]
\newtheorem{proposition}[theorem]{Proposition}
\newtheorem{lemma}[theorem]{Lemma}
\newtheorem{corollary}[theorem]{Corollary}

\newtheorem*{MainResult1}{Theorem~\ref{thm:Gamma1MeetIrred}}
\newtheorem*{MainResult2}{Theorem~\ref{thm:bk=jdt}}
\newtheorem*{MainResult3}{Theorem~\ref{thm:newprorow}}
\newtheorem*{MainResult4}{Theorem~\ref{thm:proequalspro}}
\newtheorem*{MainResult5}{Theorem~\ref{thm:prorow}}
\newtheorem*{MainResult6}{Corollary~\ref{thm:res}}
\newtheorem*{MainResult7}{Corollary~\ref{cor:rowres}}
\newtheorem*{MainResult8}{Corollary~\ref{cor:prorow}}
\newtheorem*{MainResult9}{Corollary~\ref{cor:cartesian}}

\theoremstyle{definition}
\newtheorem{definition}[theorem]{Definition}

\newtheorem*{DefTogOrder}{Definition~\ref{TogOrder}}
\newtheorem*{DefTogPro}{Definition~\ref{def:TogPro}}
\newtheorem{example}[theorem]{Example}

\newtheorem{remark}[theorem]{Remark}
\numberwithin{equation}{section}

\title{Rowmotion and increasing labeling promotion}
\author{Kevin Dilks, Jessica Striker, Corey Vorland}

\begin{document}
\maketitle

\begin{abstract}
In 2012, N.\ Williams and the second author showed that on order ideals of ranked partially ordered sets (posets), rowmotion is conjugate to (and thus has the same orbit structure as) a different toggle group action, which in special cases is equivalent to promotion on linear extensions of posets constructed from two chains. In 2015, O.\ Pechenik and the first and second authors extended these results to show that increasing tableaux under $K$-promotion naturally corresponds to order ideals in a product of three chains under a toggle group action conjugate to rowmotion they called hyperplane promotion.

In this paper, we generalize these results to the setting of arbitrary increasing labelings of any finite poset with given restrictions on the labels. We define a generalization of $K$-promotion in this setting and show it corresponds to a toggle group action we call toggle-promotion on order ideals of an associated poset. When the restrictions on labels are particularly nice (for example, specifying a global bound on all labels used), we show that toggle-promotion is conjugate to rowmotion. Additionally, we show that any poset that can be nicely embedded into a Cartesian product has a natural toggle-promotion action conjugate to rowmotion.
\end{abstract}

\tableofcontents

\section{Introduction}
In this paper, we define a natural generalization of M.-P.\ Sch\"utzenberger's \emph{promotion} operator in the setting of \emph{increasing labelings} on any finite poset. We relate this generalized promotion to 
\emph{rowmotion}, denoted $\row$, defined on any order ideal $I$ of a finite partially ordered set $P$ as the order ideal generated by the minimal elements of $P\setminus I$. We do this by showing our generalized promotion is conjugate to {rowmotion} in the \emph{toggle group} on the associated poset of meet-irreducibles of these increasing labelings under their natural partial order. 

This paper builds upon work of N.\ Williams and the second author~\cite{SW2012} regarding promotion and rowmotion on posets with a two-dimensional lattice projection and subsequent work of O.\ Pechenik with the first and second authors~\cite{DPS2015} which $n$-dimensionalized this result, relating \emph{$K$-promotion} on \emph{increasing tableaux} to rowmotion on the product of three chains poset. 

We describe our main results in Subsection~\ref{subsec:main_results} and then give a brief history of promotion and rowmotion in Subsection~\ref{sec:background}, with a focus on the motivating results from these two papers.

\subsection{Main Results}
\label{subsec:main_results}
Throughout this paper, let $P$ be a finite partially ordered set (poset). 

\begin{definition}
We say that a function $f:P\rightarrow \mathbb{Z}$ is an \emph{increasing labeling} if $p_1< p_2$ in $P$ implies that $f(p_1)< f(p_2)$ (with the usual total ordering on the integers). We will be interested in sets of increasing labelings on $P$ given a restriction function
$R:P\mapsto \mathcal{P}(\mathbb{Z})$ indicating which labels each poset element is allowed to attain (where $\mathcal{P}(\mathbb{Z})$ is the power set of $\mathbb{Z}$). 
We require $R(p)$ to be nonempty and finite for each $p\in P$. Call the set of such increasing labelings $\inc{R}{P}$. 
\end{definition}

See Figures~\ref{fig:TogIsProEx_intro} and \ref{fig:IncPro} for examples.

\begin{remark}
Up to conventions of increasing versus decreasing, increasing labelings can also be thought of as \emph{strict $P$-partitions with restricted parts}. We use the terminology of increasing labelings rather than $P$-partitions since we are generalizing from increasing tableaux rather than from integer partitions. See Remarks~\ref{remark:propp}, \ref{rem:Ppart}, and \ref{rem:Ppart2} for more on the connection to $P$-partitions.
\end{remark}

We consider a natural partial order on $\inc{R}{P}$, where $f\leq g$ if and only if $f(p)\leq g(p)$ for all $p\in P$. 
This is a distributive lattice, so we may apply Birkhoff's Representation Theorem 
to obtain a representation in terms of order ideals in a poset. A subset $I$ of $P$ is called an \emph{order ideal} if for any $t \in I$ and $s \le t$ in $P$, then $s \in I$. Let $J(P)$ denote the set of order ideals of $P$.

Our \textbf{first main result} describes the poset of meet irreducibles of the distributive lattice $\inc{R}{P}$, which we denote as $\Gamma(P,R)$. This gives a bijection between the increasing labelings $\inc{R}{P}$ and the order ideals  $J(\Gamma(P,R))$. 
We construct $\Gamma(P,R)$ in Section~\ref{subsec:gamma1}, but note here that the elements of $\Gamma(P,R)$ are certain ordered pairs $(p,k)$, where $p\in P$ and $k\in\mathbb{Z}$. 

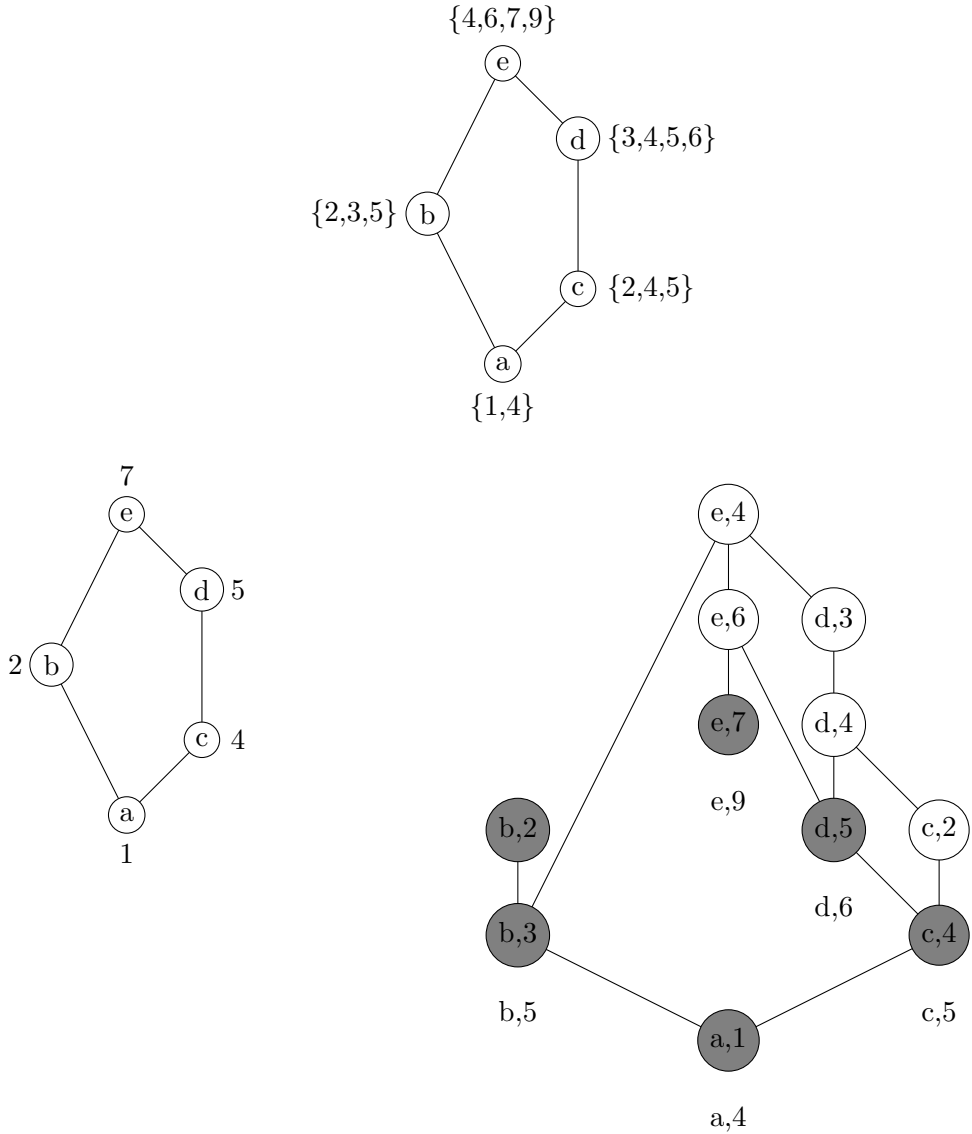
\begin{figure}[hbtp]
\begin{tikzpicture}
\begin{scope}[yshift=5cm,xshift=5cm]
\coordinate (a) at (0,0);
\coordinate (b) at (-1,2);
\coordinate (c) at (1,1);
\coordinate (d) at (1,3);
\coordinate (e) at (0,4);
\draw[] (a) -- (b) -- (e) -- (d) -- (c) -- cycle ;
\node[fill=white,draw,circle,inner sep=.5ex] at (a) {a};
\node[fill=white,draw,circle,inner sep=.5ex] at (b) {b};
\node[fill=white,draw,circle,inner sep=.5ex] at (c) {c};
\node[fill=white,draw,circle,inner sep=.5ex] at (d) {d};
\node[fill=white,draw,circle,inner sep=.5ex] at (e) {e};
\node [below=.25cm] at (a) {\{1,4\}};
\node [left=.25cm] at (b) {\{2,3,5\}};
\node [right=.25cm] at (c) {\{2,4,5\}};
\node [right=.25cm] at (d) {\{3,4,5,6\}};
\node [above=.25cm] at (e) {\{4,6,7,9\}};
\end{scope}

\begin{scope}[yshift=-1cm,xshift=0cm]
\coordinate (a) at (0,0);
\coordinate (b) at (-1,2);
\coordinate (c) at (1,1);
\coordinate (d) at (1,3);
\coordinate (e) at (0,4);
\draw[] (a) -- (b) -- (e) -- (d) -- (c) -- cycle ;
\node[fill=white,draw,circle,inner sep=.5ex] at (a) {a};
\node[fill=white,draw,circle,inner sep=.5ex] at (b) {b};
\node[fill=white,draw,circle,inner sep=.5ex] at (c) {c};
\node[fill=white,draw,circle,inner sep=.5ex] at (d) {d};
\node[fill=white,draw,circle,inner sep=.5ex] at (e) {e};
\node [below=.25cm] at (a) {1};
\node [left=.25cm] at (b) {2};
\node [right=.25cm] at (c) {4};
\node [right=.25cm] at (d) {5};
\node [above=.25cm] at (e) {7};
\end{scope}

\begin{scope}[yshift=-4cm,xshift=8cm,scale=1.4]
\coordinate (a1) at (0,0);
\coordinate (a4) at (0,-.75);
\coordinate (b2) at (-2,2);
\coordinate (b3) at (-2,1);
\coordinate (b5) at (-2,0.25);
\coordinate (c2) at (2,2);
\coordinate (c4) at (2,1);
\coordinate (c5) at (2,0.25);
\coordinate (d3) at (1,4);
\coordinate (d4) at (1,3);
\coordinate (d5) at (1,2);
\coordinate (d6) at (1,1.25);
\coordinate (e4) at (0,5);
\coordinate (e6) at (0,4);
\coordinate (e7) at (0,3);
\coordinate (e9) at (0,2.25);
\draw[] (b2) -- (b3);
\draw[] (c2) -- (c4);
\draw[] (d3) -- (d4) -- (d5);
\draw[] (e4) -- (e6) -- (e7);
\draw[] (a1) -- (b3);
\draw[] (a1) -- (c4);
\draw[] (b3) -- (e4);
\draw[] (c4) -- (d5);
\draw[] (c2) -- (d4);
\draw[] (d5) -- (e6);
\draw[] (d3) -- (e4);
\node[fill=gray,draw,circle,inner sep=.5ex] at (a1) {a,1};
\node[] at (a4) {a,4};
\node[fill=gray,draw,circle,inner sep=.5ex] at (b2) {b,2};
\node[fill=gray,draw,circle,inner sep=.5ex] at (b3) {b,3};
\node[] at (b5) {b,5};
\node[fill=white,draw,circle,inner sep=.5ex] at (c2) {c,2};
\node[fill=gray,draw,circle,inner sep=.5ex] at (c4) {c,4};
\node[] at (c5) {c,5};
\node[fill=white,draw,circle,inner sep=.5ex] at (d3) {d,3};
\node[fill=white,draw,circle,inner sep=.5ex] at (d4) {d,4};
\node[fill=gray,draw,circle,inner sep=.5ex] at (d5) {d,5};
\node[] at (d6) {d,6};
\node[fill=white,draw,circle,inner sep=.5ex] at (e4) {e,4};
\node[fill=white,draw,circle,inner sep=.5ex] at (e6) {e,6};
\node[fill=gray,draw,circle,inner sep=.5ex] at (e7) {e,7};
\node[] at (e9) {e,9};
\end{scope}
\end{tikzpicture}
\caption{Top: A poset $P$ with restriction function $R$ given by the sets of possible labels at each vertex; Left: An increasing labeling in $\inc{R}{P}$; Right: The  poset $\Gamma(P,R)$ with elements shaded to form the order ideal that corresponds to the increasing labeling on the left under the bijection of Theorem~\ref{thm:Gamma1MeetIrred}.}
\label{fig:TogIsProEx_intro}
\end{figure}

 \begin{MainResult1}
$\Gamma(P,R)$ is isomorphic to the dual of the lattice of meet irreducibles of $\inc{R}{P}$. Therefore, order ideals of $\Gamma(P,R)$ are in bijection with the increasing labelings $\inc{R}{P}$.
\end{MainResult1}

Next, we define increasing labeling promotion, which we denote as $\incpro$, via {generalized Bender-Knuth involutions}.
Our \textbf{second main result}, Theorem~\ref{thm:proequalspro}, shows the bijection of Theorem~\ref{thm:Gamma1MeetIrred} equivariantly takes $\incpro$ to a natural {toggle group action} $\togpro_{H_{\Gamma}}$ on order ideals of $\Gamma(P,R)$. We define $\incpro$ and $\togpro_{H_{\Gamma}}$ below and then state Theorem~\ref{thm:proequalspro} relating them.  

\begin{definition}
\label{def:bkR}
Let $p\in P$ and $f\in\inc{R}{P}$. 
For each $i\in\mathbb{Z}$,
define the \emph{$i$th generalized Bender-Knuth involution} $\rho_i:\inc{R}{P}\rightarrow \inc{R}{P}$  as follows:

\[\rho_i(f)(p)=\begin{cases}
R(p)_{>i} &  f(p)=i \mbox{ and the resulting labeling is still in } \inc{R}{P}\\
i &  f(p)=R(p)_{>i} \mbox{ and the resulting labeling is still in } \inc{R}{P}\\
f(p) &\mbox{otherwise,}
\end{cases}\]
where $R(p)_{>i}$ denotes the smallest label in $R(p)$ that is larger than $i$.
That is, $\rho_i$ changes $i$ to $R(p)_{>i}$ and/or $R(p)_{>i}$ to $i$ {wherever possible}. 

Define \emph{increasing labeling promotion} as $\incpro(f)=\cdots\circ\rho_{3}\circ\rho_{2}\circ\rho_{1}\circ\cdots(f)$.
\end{definition}

See Figure~\ref{fig:IncPro} for an example. Note
since $P$ is finite and $R(p)$ is finite for each $p\in P$, the infinite product of the $\rho_i$ reduces to a finite product.

\begin{remark}
\label{remark:propp}
J.\ Propp has also defined a generalized promotion on $P$-partitions using local involutions \cite{Propp2015},
as an analogue of \emph{piecewise-linear promotion} \cite{EP2014}. 
Propp's notion is a piecewise-linear lift of our Definition~\ref{def:bkR}, restricted to
 the integer points in a dilation of the order polytope.
\end{remark}

\begin{figure}[htbp]
\scalebox{.9}{
\resizebox{\textwidth}{!}{
\begin{tikzpicture}[xshift=-5cm]

\node at (0,0) {
\begin{tikzpicture}
\coordinate (a) at (0,0);
\coordinate (b) at (1,1);
\coordinate (c) at (-1,1);
\coordinate (d) at (-2,2);
\coordinate (e) at (-3,1);
\draw[] (b) -- (a) -- (c) -- (d) -- (e) ;
\node[fill=white,draw,circle,inner sep=.75ex] at (a) {1};
\node[fill=white,draw,circle,inner sep=.75ex] at (b) {3};
\node[fill=white,draw,circle,inner sep=.75ex] at (c) {3};
\node[fill=white,draw,circle,inner sep=.75ex] at (d) {5};
\node[fill=white,draw,circle,inner sep=.75ex] at (e) {2};
\end{tikzpicture}
};
\node at (6,0) {
\begin{tikzpicture}
\coordinate (a) at (0,0);
\coordinate (b) at (1,1);
\coordinate (c) at (-1,1);
\coordinate (d) at (-2,2);
\coordinate (e) at (-3,1);
\draw[] (b) -- (a) -- (c) -- (d) -- (e) ;
\node[fill=white,draw,circle,inner sep=.75ex] at (a) {2};
\node[fill=white,draw,circle,inner sep=.75ex] at (b) {3};
\node[fill=white,draw,circle,inner sep=.75ex] at (c) {3};
\node[fill=white,draw,circle,inner sep=.75ex] at (d) {5};
\node[fill=white,draw,circle,inner sep=.75ex] at (e) {1};
\end{tikzpicture}
};
\node at (12,0) {
\begin{tikzpicture}
\coordinate (a) at (0,0);
\coordinate (b) at (1,1);
\coordinate (c) at (-1,1);
\coordinate (d) at (-2,2);
\coordinate (e) at (-3,1);
\draw[] (b) -- (a) -- (c) -- (d) -- (e) ;
\node[fill=white,draw,circle,inner sep=.75ex] at (a) {2};
\node[fill=white,draw,circle,inner sep=.75ex] at (b) {3};
\node[fill=white,draw,circle,inner sep=.75ex] at (c) {3};
\node[fill=white,draw,circle,inner sep=.75ex] at (d) {5};
\node[fill=white,draw,circle,inner sep=.75ex] at (e) {1};
\end{tikzpicture}
};
\node at (12,-4) {
\begin{tikzpicture}
\coordinate (a) at (0,0);
\coordinate (b) at (1,1);
\coordinate (c) at (-1,1);
\coordinate (d) at (-2,2);
\coordinate (e) at (-3,1);
\draw[] (b) -- (a) -- (c) -- (d) -- (e) ;
\node[fill=white,draw,circle,inner sep=.75ex] at (a) {2};
\node[fill=white,draw,circle,inner sep=.75ex] at (b) {4};
\node[fill=white,draw,circle,inner sep=.75ex] at (c) {4};
\node[fill=white,draw,circle,inner sep=.75ex] at (d) {5};
\node[fill=white,draw,circle,inner sep=.75ex] at (e) {1};
\end{tikzpicture}
};
\node at (6,-4) {
\begin{tikzpicture}
\coordinate (a) at (0,0);
\coordinate (b) at (1,1);
\coordinate (c) at (-1,1);
\coordinate (d) at (-2,2);
\coordinate (e) at (-3,1);
\draw[] (b) -- (a) -- (c) -- (d) -- (e) ;
\node[fill=white,draw,circle,inner sep=.75ex] at (a) {2};
\node[fill=white,draw,circle,inner sep=.75ex] at (b) {5};
\node[fill=white,draw,circle,inner sep=.75ex] at (c) {4};
\node[fill=white,draw,circle,inner sep=.75ex] at (d) {5};
\node[fill=white,draw,circle,inner sep=.75ex] at (e) {1};
\end{tikzpicture}
};
\draw[thick,->] (2.5,0.5) -- (4,0.5) node[midway,above] {$\mbox{\huge $\rho_1$}$};
\draw[thick,->] (8.5,0.5) -- (10,0.5) node[midway,above] {$\mbox{\huge $\rho_2$}$};
\draw[thick,->] (12,-2) -- (12,-3) node[midway,right] {$\mbox{\huge $\rho_3$}$};
\draw[thick,->] (10,-3) -- (8.5,-3) node[midway,above] {$\mbox{\huge $\rho_4$}$};
\draw[thick,->] (0,-2) to [out=270,in=180] (2.5,-3);
\node at (.75,-3.5) {$\mbox{\huge $\incpro$}$};
\end{tikzpicture}
}
}
\caption{An increasing labeling in $\inc{5}{P}$ (which is $\inc{R}{P}$ with restriction function $R$ induced by requiring that all labels be in $\{1,2,3,4,5\}$) and the composition of generalized Bender-Knuth involutions giving $\incpro$.}
\label{fig:IncPro}
\end{figure}
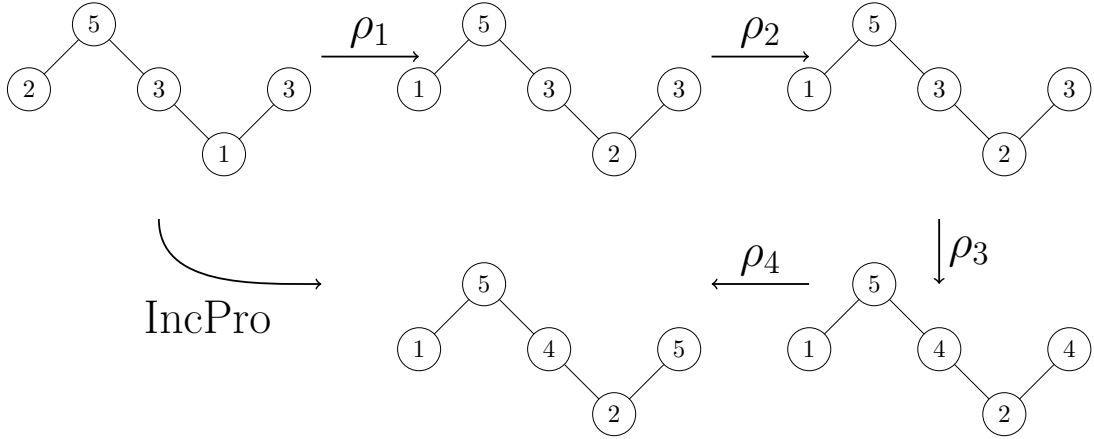

We define the toggle group action $\togpro_{H_{\Gamma}}$ below, using the following definitions. 
Given $p\in P$, its \emph{toggle} $t_p$ acting on an order ideal $I$ 
gives the symmetric difference of $p$ and $I$ if the resulting set is still an order ideal. For $p_1,p_2\in P$, $p_1\lessdot p_2$ means $p_2$ \emph{covers} $p_1$, in the sense that $p_1<p_2$ and there is no $p\in P$ such that $p_1<p<p_2$.

\begin{DefTogOrder}
We say that a function $H:P\rightarrow \mathbb{Z}$ is a \emph{toggle order} if $p_1\lessdot p_2$ implies $H(p_1)\neq H(p_2)$. 
Given a toggle order $H$, define $T_H^i$ to be the toggle group action that is the product of all $t_p$ for $p\in P$ such that $H(p)=i$.
\end{DefTogOrder}

\begin{DefTogPro}
We say that \emph{toggle-promotion} with respect to a toggle order $H$, denoted $\togpro_H$, is the toggle group action given by \[ \ldots T_H^{-2}T_H^{-1}T_H^0 T_H^1 T_H^2\ldots \]
\end{DefTogPro}

\begin{MainResult4}
Let $H_{\Gamma}:\Gamma(P,R)\rightarrow\mathbb{Z}$ be the toggle order taking $(p,k)$ to $k$.
Then $\inc{R}{P}$ under $\incpro$ is in equivariant bijection with $J(\Gamma(P,R))$ under $\togpro_{H_{\Gamma}}$.
\end{MainResult4}

See Figure~\ref{fig:TogIsProEx_intro} and Example~\ref{ex:proequalspro}.

\smallskip
Our \textbf{third main result} generalizes results of \cite{SW2012} and \cite{DPS2015} to show the equivariance of {rowmotion} and toggle-promotion $\togpro_H$ on any poset  with a \emph{column toggle order} $H$ (see Definition~\ref{def:ColTog}).
\begin{MainResult3}
Let $H$ be a column toggle order of $P$. Then the toggle group action $\togpro_{H}$ on $J(P)$ is conjugate to $\row$ on $J(P)$. 
\end{MainResult3}

In the special case of a column toggle order, we can use Theorem~\ref{thm:newprorow} to map equivariantly between increasing labeling promotion on $\inc{R}{P}$ and rowmotion on order ideals of $\Gamma(P,R)$. 
\begin{MainResult5}
When $H_{\Gamma}$ is a column toggle order,
there is an equivariant bijection between $\inc{R}{P}$ under $\incpro$ and order ideals in $\Gamma(P,R)$ under $\row$.
\end{MainResult5}

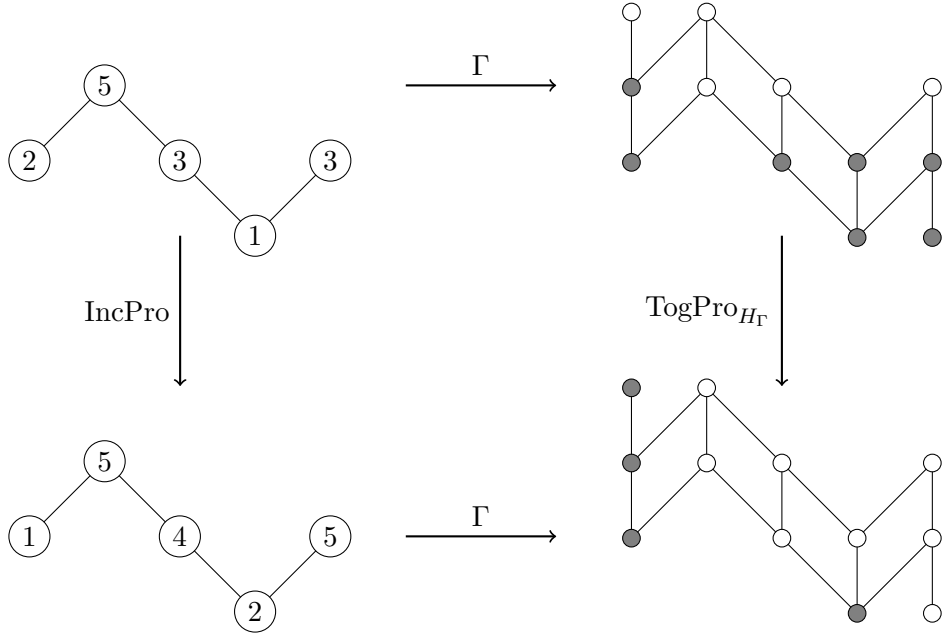
\begin{figure}[htbp]
\begin{tikzpicture}
\node at (8,0) {
\begin{tikzpicture}
\coordinate (a0) at (0,-1);
\coordinate (a) at (0,0);
\coordinate (a1) at (0,1);
\coordinate (d-1) at (1,-1);
\coordinate (d0) at (1,0);
\coordinate (d) at (1,1);
\coordinate (d1) at (1,2);
\coordinate (c0) at (-1,0);
\coordinate (c) at (-1,1);
\coordinate (c1) at (-1,2);
\coordinate (e0) at (-2,1);
\coordinate (e) at (-2,2);
\coordinate (e1) at (-2,3);
\coordinate (b0) at (-3,0);
\coordinate (b) at (-3,1);
\coordinate (b1) at (-3,2);
\coordinate (b2) at (-3,3);
\draw[] (d) -- (a) -- (c) -- (e) -- (b) ;
\draw[] (d1) -- (a1) -- (c1) -- (e1) -- (b1) ;
\draw (a1) -- (a);
\draw (d1) -- (d);
\draw (c1) -- (c);
\draw (e1) -- (e);
\draw (b1) -- (b);
\draw (d0) -- (d);
\draw (b2) -- (b1);
\node[inner sep=.5ex] at (a0) {};
\node[fill=gray,draw,circle,inner sep=.5ex] at (a) {};
\node[fill=white,draw,circle,inner sep=.5ex] at (a1) {};
\node[inner sep=.5ex] at (d-1) {};
\node[fill=white,draw,circle,inner sep=.5ex] at (d0) {};
\node[fill=white,draw,circle,inner sep=.5ex] at (d) {};
\node[fill=white,draw,circle,inner sep=.5ex] at (d1) {};
\node[inner sep=.5ex] at (c0) {};
\node[fill=white,draw,circle,inner sep=.5ex] at (c) {};
\node[fill=white,draw,circle,inner sep=.5ex] at (c1) {};
\node[inner sep=.5ex] at (e0) {};
\node[fill=white,draw,circle,inner sep=.5ex] at (e) {};
\node[fill=white,draw,circle,inner sep=.5ex] at (e1) {};
\node[inner sep=.5ex] at (b0) {};
\node[fill=gray,draw,circle,inner sep=.5ex] at (b) {};
\node[fill=gray,draw,circle,inner sep=.5ex] at (b1) {};
\node[fill=gray,draw,circle,inner sep=.5ex] at (b2) {};
\end{tikzpicture}
};
\node at (8,5) {
\begin{tikzpicture}
\coordinate (a0) at (0,-1);
\coordinate (a) at (0,0);
\coordinate (a1) at (0,1);
\coordinate (d-1) at (1,-1);
\coordinate (d0) at (1,0);
\coordinate (d) at (1,1);
\coordinate (d1) at (1,2);
\coordinate (c0) at (-1,0);
\coordinate (c) at (-1,1);
\coordinate (c1) at (-1,2);
\coordinate (e0) at (-2,1);
\coordinate (e) at (-2,2);
\coordinate (e1) at (-2,3);
\coordinate (b0) at (-3,0);
\coordinate (b) at (-3,1);
\coordinate (b1) at (-3,2);
\coordinate (b2) at (-3,3);
\draw[] (d) -- (a) -- (c) -- (e) -- (b) ;
\draw[] (d1) -- (a1) -- (c1) -- (e1) -- (b1) ;
\draw (a1) -- (a);
\draw (d1) -- (d);
\draw (c1) -- (c);
\draw (e1) -- (e);
\draw (b1) -- (b);
\draw (d0) -- (d);
\draw (b2) -- (b1);
\node[inner sep=.5ex] at (a0) {};
\node[fill=gray,draw,circle,inner sep=.5ex] at (a) {};
\node[fill=gray,draw,circle,inner sep=.5ex] at (a1) {};
\node[inner sep=.5ex] at (d-1) {};
\node[fill=gray,draw,circle,inner sep=.5ex] at (d0) {};
\node[fill=gray,draw,circle,inner sep=.5ex] at (d) {};
\node[fill=white,draw,circle,inner sep=.5ex] at (d1) {};
\node[inner sep=.5ex] at (c0) {};
\node[fill=gray,draw,circle,inner sep=.5ex] at (c) {};
\node[fill=white,draw,circle,inner sep=.5ex] at (c1) {};
\node[inner sep=.5ex] at (e0) {};
\node[fill=white,draw,circle,inner sep=.5ex] at (e) {};
\node[fill=white,draw,circle,inner sep=.5ex] at (e1) {};
\node[inner sep=.5ex] at (b0) {};
\node[fill=gray,draw,circle,inner sep=.5ex] at (b) {};
\node[fill=gray,draw,circle,inner sep=.5ex] at (b1) {};
\node[fill=white,draw,circle,inner sep=.5ex] at (b2) {};
\end{tikzpicture}
};
\node at (0,0) {
\begin{tikzpicture}
\coordinate (a) at (0,0);
\coordinate (b) at (1,1);
\coordinate (c) at (-1,1);
\coordinate (d) at (-2,2);
\coordinate (e) at (-3,1);
\draw[] (b) -- (a) -- (c) -- (d) -- (e) ;
\node[fill=white,draw,circle,inner sep=.5ex] at (a) {2};
\node[fill=white,draw,circle,inner sep=.5ex] at (b) {5};
\node[fill=white,draw,circle,inner sep=.5ex] at (c) {4};
\node[fill=white,draw,circle,inner sep=.5ex] at (d) {5};
\node[fill=white,draw,circle,inner sep=.5ex] at (e) {1};
\end{tikzpicture}
};
\node at (0,5) {
\begin{tikzpicture}
\coordinate (a) at (0,0);
\coordinate (b) at (1,1);
\coordinate (c) at (-1,1);
\coordinate (d) at (-2,2);
\coordinate (e) at (-3,1);
\draw[] (b) -- (a) -- (c) -- (d) -- (e) ;
\node[fill=white,draw,circle,inner sep=.5ex] at (a) {1};
\node[fill=white,draw,circle,inner sep=.5ex] at (b) {3};
\node[fill=white,draw,circle,inner sep=.5ex] at (c) {3};
\node[fill=white,draw,circle,inner sep=.5ex] at (d) {5};
\node[fill=white,draw,circle,inner sep=.5ex] at (e) {2};
\end{tikzpicture}
};
\draw[thick,->] (3,0) -- (5,0) node[midway,above] {$\Gamma$};
\draw[thick,->] (3,6) -- (5,6) node[midway,above] {$\Gamma$};
\draw[thick,->] (0,4) -- (0,2) node[midway,left] {$\incpro$};
\draw[thick,->] (8,4) -- (8,2) node[midway,left] {$\togpro_{H_{\Gamma}}$};
\end{tikzpicture}
\caption{Correspondence between increasing labelings with labels at most 5 under $\incpro$ and order ideals in the associated $\Gamma$ poset under $\togpro_{H_{\Gamma}}$.}
\label{fig:ComDiag}
\end{figure}

Several of our main results given above specialize nicely in the case of
increasing labelings with labels restricted to the range $\{1,\ldots,q\}$, which we denote as 
$\inc{q}{P}$. 
Let $\Gamma(P,q)$ be the poset $\Gamma(P,R)$ for the induced restriction function.
The following is a corollary of Theorem~\ref{thm:prorow}.
\begin{MainResult8}
There is an equivariant bijection between $\inc{q}{P}$ under $\incpro$ and order ideals in $\Gamma(P,q)$ under $\row$.
\end{MainResult8}

This result is a generalization of work of O.\ Pechenik with the first and second authors giving an equivariant bijection between \emph{$K$-promotion} on \emph{increasing tableaux} and rowmotion on order ideals of the product of three chains poset~\cite{DPS2015}. We give more details on this in Subsection~\ref{sec:background}. (See Figure~\ref{fig:ComDiag})

\smallskip
Another case where we may define a column toggle order to obtain a toggle group action conjugate to rowmotion arises from Cartesian products; see Subsection~\ref{sec:app} for relevant definitions. The following is a corollary of Theorem~\ref{thm:newprorow}.

\begin{MainResult9}
Let $P$ be a ranked poset with a Cartesian embedding into ranked posets $(P_1,P_2)$. Let $H$ map the element of $P$ embedded at coordinate $(p_1,p_2)$ to the difference of ranks $\mathrm{rk}_{P_1}(p_1) - \mathrm{rk}_{P_2}(p_2)$. Then $\togpro_H$ on $J(P)$ is conjugate to $\row$ on $J(P)$. 
\end{MainResult9}

Finally, in the case of $\inc{q}{P}$, we also define an analogue of  \emph{jeu de taquin} promotion in Definition~\ref{def:jdt}, which we denote as $\jdtpro$. The \textbf{last main result} we mention here shows this is the same action as $\incpro$.

\begin{MainResult2}
For $f\in\inc{q}{P}$, $\incpro(f)=\jdtpro(f)$. 
\end{MainResult2}

We apply this correspondence to prove the following instance of the \emph{resonance} phenomenon, which, informally speaking, occurs when an action projects to a cyclic action, generally of smaller order. This is an analogue of \cite[Theorem 2.2]{DPS2015} in the case of increasing tableaux. See the next subsection for the precise definition of resonance and a discussion of the analogous result from \cite{DPS2015}. Let $\mathrm{Con}(f)$ 
denote the \emph{binary content} of $f\in\inc{q}{P}$, 
that is, the length $q$ vector whose $i$th entry is $1$ if $f(p)=i$ for some $p\in P$ and $0$ otherwise. 

\begin{MainResult6}
$(\inc{q}{P},\langle\incpro\rangle,\mbox{\rm Con})$ exhibits resonance with frequency $q$.
\end{MainResult6}

This corollary and Corollary~\ref{cor:prorow} together imply a new resonance result on rowmotion as well. 

\begin{MainResult7}
Let $\varphi$ denote the map from an order ideal in $\Gamma(P,q)$ to the corresponding increasing labeling on $P$, and let $d$ be the toggle group element conjugating $\row$ to $\togpro_{H_{\Gamma}}$. Then
$(J(\Gamma(P,q)),\langle\row\rangle,\mbox{\rm Con}\circ\varphi\circ d)$ exhibits resonance with frequency~$q$. 
\end{MainResult7}

\textbf{The paper is structured as follows}. In Subsection~\ref{subsec:gamma1}, we construct $\Gamma(P,R)$ and prove Theorem~\ref{thm:Gamma1MeetIrred} relating $\inc{R}{P}$ and $J(\Gamma(P,R))$. In Subsection~\ref{subsec:q}, we restrict our attention to $\inc{q}{P}$, and in Subsection~\ref{subsec:ranked}, we discuss weakly increasing labelings and prove an analogue of Theorem~\ref{thm:Gamma1MeetIrred} in this setting.
In Section~\ref{sec:Promotion}, we define $\jdtpro$ and prove Theorem~\ref{thm:bk=jdt} relating $\jdtpro$ and $\incpro$ as well as Corollary~\ref{thm:res} on resonance. In Section~\ref{sec:eq_bij}, we establish the equivariance of the bijection between increasing labelings under $\incpro$ and the corresponding order ideals under rowmotion. In Subsection~\ref{row=pro}, we prove Theorem~\ref{thm:newprorow} on the conjugacy of toggle-promotion $\togpro_H$ and rowmotion. In Subsection~\ref{sec:app} we apply Theorem~\ref{thm:newprorow} to $\Gamma(P,R)$ and Cartesian embeddings, yielding Corollary~\ref{cor:cartesian}. In Subsection~\ref{pro=pro}, we prove Theorem~\ref{thm:prorow} and Corollaries~\ref{cor:prorow} and~\ref{cor:rowres}. Finally, in Section~\ref{sec:converse}, we discuss an example in which we can give an inverse to Theorem~\ref{thm:prorow}.

But first, in the next subsection, we put our results in context by giving a brief history of motivating and analogous work.

\subsection{A brief history of promotion and rowmotion}
\label{sec:background}
In this section, we give some background on promotion and rowmotion with a focus on results that give realms in which these maps correspond to each other. This paper generalizes these results to a broader setting.

\subsubsection{Promotion on standard Young tableaux and linear extensions}
Promotion is a natural action defined by M.-P.\ Sch\"utzenberger on standard Young tableaux and, more generally, linear extensions of finite poset \cite{Sch1972}, arising from study of evacuation and the RSK correspondence. A \emph{linear extension} of a poset $P$ is a bijective function $f:P\rightarrow \{1,\ldots,n\}$ where $|P|=n$ such that if $p_1< p_2$ in $P$ then $f(p_1)< f(p_2)$. In the notation of this paper, a linear extension is an increasing labeling in $\inc{n}{P}$ such that each label value is used exactly once. Let $\mathcal{L}(P)$ denote the set of linear extensions of $P$.  Suppose $f\in\mathcal{L}(P)$, then the \emph{promotion} of $f$, denoted $\pro(f)$, is found as follows. We begin by deleting the label 1. We then slide down the smallest label of all covers of the now unlabeled element to replace the removed label $1$; this is called a \emph{jeu de taquin} slide. This jeu de taquin sliding process continues with the new unlabeled element until the unlabeled element is maximal; we then label this with $n+1$. By subtracting 1 from every label, we obtain a new linear extension, which is $\pro(f)$.

This is not the only way to view promotion on a linear extension; it can also be defined using a sequence of involutions, which are a special case of involutions introduced by E.~Bender and D.~Knuth on semistandard Young tableaux~\cite{BK1972}. Let the action of the \emph{$i$th Bender-Knuth involution} $\rho_i$ on $f\in\mathcal{L}(P)$ be as follows: swap the labels $i$ and $i+1$ if the result is a linear extension, otherwise do nothing. 
Then $\pro(f)=\rho_{n-1}\circ\cdots\circ\rho_2\circ\rho_1$. 
In the notation of this paper, if $f\in\inc{n}{P}$ is also in $\mathcal{L}(P)$, our generalized Bender-Knuth involutions of Definition~\ref{def:bkR} reduce to these Bender-Knuth involutions and $\incpro$ equals $\pro$.
Promotion has many beautiful properties and  significant applications in representation theory. 
See R.\ Stanley's survey \cite{Stanley2009} for many of these properties, including further history and details on promotion via Bender-Knuth involutions.

\subsubsection{Rowmotion and the toggle group}
Rowmotion is an action originally defined on hypergraphs by P.~Duchet \cite{Duchet1974} and generalized to order ideals of an arbitrary finite poset by A.\ Brouwer and A.\ Schrijver~\cite{BS1974}. Given $I\in J(P)$, \emph{rowmotion} on $I$, denoted $\row(I)$, is the order ideal generated by the minimal elements of $P\setminus I$.  Rowmotion has recently generated significant interest as a prototypical action in the emerging subfield of dynamical algebraic combinatorics; see \cite{SW2012} for a detailed history and \cite{BR2011, Brouwer1975, CHHM2017, DF1990, DPS2015, EP2013, EP2014, GP2017, GR2015, GR2016, Joseph2017, Panyushev2009, PR2015, Rush2016, RS2013, RW2015, Striker2015, Striker2016, Striker2017, Vorland2017} for more recent developments.  

In \cite{BS1974}, Brouwer and Schrijver studied the order of rowmotion on a product of two chains poset, $\mathbf{a} \times \mathbf{b}$, discovering that $J(\mathbf{a} \times \mathbf{b})$ has order $a+b$ under rowmotion. D.\ Fon-der-Flaass refined this further with a result on the length of any orbit of $J(\mathbf{a} \times \mathbf{b})$ under rowmotion \cite{Fon-der-Flaass1993}. In \cite{Stanley2009}, Stanley showed there exists an equivariant bijection between linear extensions of two disjoint chains $\mathbf{a} \oplus \mathbf{b}$ (or equivalently, standard Young tableaux of disjoint skew shape) under promotion and $J(\mathbf{a} \times \mathbf{b})$ under rowmotion.

\begin{theorem}[\protect{\cite{Stanley2009}}]
\label{thm:skew}
$J\left(\mathbf{a} \times \mathbf{b}\right)$ under $\row$ is in equivariant bijection with $\mathcal{L}\left(\mathbf{a} \oplus \mathbf{b}\right)$ under $\pro$.
\end{theorem}

Another instance where promotion on linear extensions and rowmotion are related is an equivariant bijection between linear extensions of the product of two chains poset $\mathbf{2} \times \mathbf{n}$ 
under promotion (or alternatively, rectangular, two-row standard Young tableaux under promotion) and order ideals of the type $A_{n-1}$ positive root poset $\Phi^+(A_{n-1})$ under rowmotion. This is a restatement of the Type $A$ case of a result of D.\ Armstrong, C.\ Stump, and H.\ Thomas in \cite{AST2013}. 

\begin{theorem}[\protect{see \cite[Theorem 3.10]{SW2012}}]
\label{thm:catalan}
$J\left(\Phi^+(A_{n-1})\right)$ under $\row$ is in equivariant bijection with $\mathcal{L}\left(\mathbf{2} \times \mathbf{n}\right)$ under $\pro$.
\end{theorem}

N.~Williams and the second author proved a general result~\cite{SW2012} relating promotion and rowmotion which recovers Theorems~\ref{thm:skew} and \ref{thm:catalan} as special cases. They used the \emph{toggle group} of P.~Cameron and D.~Fon-der-Flaass, which we describe below.

\begin{definition}
For any $p \in P$, the \textit{toggle} $t_p: J(P) \rightarrow J(P)$ is defined as follows:
\[ 
t_p(I)=
   \begin{cases} 
      I \cup \{ p \} \qquad & \text{if } p \notin I \text{ and } I \cup \{ p \} \in J(P) \\
      I \setminus \{ p \} &\text{if } p \in I \text{ and } I \setminus \{ p \} \in J(P) \\
      I &\text{otherwise.}
   \end{cases}
\]
The \emph{toggle group} of $P$ is the group generated by the $t_p$ for all $p\in P$.
\end{definition}

\begin{remark}(\protect{\cite[p.\ 546]{CF1995}})
\label{rmk:ToggleCommute}
The toggles $t_{p_1}$ and $t_{p_2}$ commute whenever neither $p_1$ nor $p_2$ covers the other.
\end{remark}

Cameron and Fon-der-Flaass showed a connection between rowmotion and toggling. Specifically, rowmotion can be performed by toggling each element of a poset from top to bottom, that is, in the reverse order of any linear extension. If the poset is ranked, this is equivalent to toggling top to bottom by ranks, or rows.

\begin{theorem}[\protect{\cite[Lemma 1]{CF1995}}]
\label{thm:linextrow}
Let $f\in\mathcal{L}(P)$.  Then $t_{f^{-1}(1)}t_{f^{-1}(2)} \cdots t_{f^{-1}(n)}$ acts as $Row$.
\end{theorem}

In~\cite{SW2012}, N.~Williams and the second author constructed a {toggle group action} that corresponds to linear extension promotion in the special cases of Theorems~\ref{thm:skew} and \ref{thm:catalan}; they named this toggle group action \emph{(order ideal-) promotion} because of this correspondence. Order ideal promotion first requires projecting to a two-dimensional lattice and defining \emph{columns}; promotion toggles poset elements from left to right by column. 
In the following theorem, they showed that order ideal promotion and rowmotion are conjugate elements in the toggle group, and thus have the same orbit structure.

\begin{theorem}[\protect{\cite[Theorem 5.2]{SW2012}}]
\label{thm:ProRowEq}
For any poset $P$ with a {two-dimensional lattice projection} (in particular, any ranked poset), there is an equivariant bijection between $J(P)$ under order ideal promotion and $J(P)$ under $Row$.
\end{theorem}

\subsubsection{$K$-promotion on increasing tableaux and rowmotion on the product of three chains}
In \cite{Pechenik2014}, O.\ Pechenik generalized Sch\"utzenberger promotion on standard Young tableaux to \emph{$K$-promotion} on \emph{increasing tableaux}, using the  \emph{$K$-jeu de taquin} of H.\ Thomas and A.\ Yong \cite{TY2009}. Increasing tableaux, a special subset of semistandard Young tableaux, first appeared in \cite{BKSTY2008} in the context of $K$-theoretic Schubert calculus. We give the definitions of increasing tableaux and $K$-promotion below.

\begin{definition}
An \textit{increasing tableau} of shape $\lambda$ is a filling of boxes of partition shape $\lambda$ with positive integers such that the entries strictly increase from left to right across rows and strictly increase from top to bottom along columns. We will use $\inc{q}{\lambda}$ to indicate the set of increasing tableaux of shape $\lambda$ with entries at most $q$.
\end{definition}

\begin{definition}
\label{def:tabjdt}
Let $T \in \inc{q}{\lambda}$. Delete all labels 1
from $T$. Consider the set of boxes that are either empty or contain 2. We simultaneously delete
each label 2 that is adjacent to an empty box and place a 2 in each empty box that is adjacent to a 2. Now consider
the set of boxes that are either empty or contain 3, and repeat the above
process. Continue until all empty boxes are located at outer corners of $\lambda$.
Finally, label those boxes $q + 1$ and then subtract 1 from each entry. The
result is the \emph{$K$-promotion of $T$}, which we denote $\kpro(T)$. Note that $\kpro(T) \in \inc{q}{\lambda}$.
\end{definition}

This is not the only way to describe $K$-promotion, however. In \cite{DPS2015}, O.\ Pechenik and the first and second authors showed that $K$-promotion can be performed using a sequence of local involutions analogous to those of Bender and Knuth for semistandard Young tableaux~\cite{BK1972}. They called these involutions $\emph{K-Bender-Knuth involutions}$, denoted by $\kbk_i$. They are the special case of our Definition \ref{def:bkR} when $P$ is of partition shape $\lambda$, since  increasing labelings on such a poset are increasing tableaux.

\begin{proposition}[\protect{\cite[Proposition 2.5]{DPS2015}}]
\label{prop:tabbk=jdt}
For $T \in \inc{q}{\lambda}$, $\kpro(T)=\kbk_{q-1} \circ \dots \circ \kbk_{1}$.
\end{proposition}

In Theorem \ref{thm:bk=jdt}, we show this result generalizes to $\inc{q}{P}$.

\smallskip
Also in  \cite{DPS2015}, the authors built on this result to give a connection between increasing tableaux of rectangular shape with entries at most $q$ and order ideals in a product of three chains poset. While the bijection between the two is straightforward, it is non-trivial  that $K$-promotion on increasing tableaux is carried equivariantly to 
a  {toggle group action} they called \emph{hyperplane promotion} on order ideals 
in the product of three chains poset. We give the relevant definitions below.

\begin{definition}[\cite{DPS2015}]
We say that an \emph{$n$-dimensional lattice projection} of a ranked poset $P$ is an order and rank preserving map $\pi : P \rightarrow \mathbb{Z}^n$, where the rank function on $\mathbb{Z}^n$ is the sum of the coordinates and $x \le y$ in $\mathbb{Z}^n$ if and only if the componentwise difference $y-x$ is in $(\mathbb{Z}_{\ge 0})^n$.
\end{definition}

\begin{definition}[\cite{DPS2015}]
\label{def:HyperToggle}
Let $P$ be a poset with an $n$-dimensional lattice projection $\pi$, and let $v = (v_1,v_2,\dots,v_n)$ where $v_j \in \{\pm 1\}$.  Let $T_{\pi, v}^i$ be the product of toggles $t_x$ for all elements $x$ of $P$ that lie on the affine hyperplane $\langle \pi(x),v \rangle=i$.  If there is no such $x$, then this is the empty product, considered to be the identity.  Define \textit{(hyperplane) promotion with respect to $\pi$ and $v$} as the toggle product Pro$_{\pi,v}=\dots T_{\pi,v}^{-2} T_{\pi,v}^{-1} T_{\pi,v}^{0} T_{\pi,v}^{1} T_{\pi,v}^{2}\dots$
\end{definition}

The authors showed one such promotion is always rowmotion.

\begin{proposition}[\protect{\cite[Proposition 3.18]{DPS2015}}]
$\mathrm{Pro}_{\pi,(1,1,\dots,1)}=\mathrm{Row}$.
\end{proposition}

This proposition and the next theorem show that any hyperplane promotion is conjugate to the more natural toggle group action of {rowmotion}.

\begin{theorem}[\protect{\cite[Theorem 3.25]{DPS2015}}]
\label{thm:DPS}
Let $P$ be a poset with an $n$-dimensional lattice projection $\pi$. Let $v$ and $w$ be vectors in $\mathbb{R}^n$ with entries in $\{\pm 1\}$. 
Then there is an equivariant bijection
between $J(P)$ under $\pro_{\pi,v}$ and $J(P)$ under $\pro_{\pi,w}$. 
\end{theorem}

\begin{theorem}[\protect{\cite[Theorem 4.4]{DPS2015}}]
\label{thm:dpsequivarbij}
$J(\mathbf{a \times b \times c})$ under $\mathrm{Row}$ is in equivariant bijection with $\inc{a+b+c-1}{a \times b}$ under $\kpro$.
\end{theorem} 

This was a second, more general setting in which rowmotion was shown to have the same orbit structure as a previously studied promotion action.
\textbf{A main purpose of the present paper} is to find an even more general realm in which rowmotion coincides with a naturally defined generalized promotion. Theorem \ref{thm:prorow} presents this main result, which can be seen as an analogue of Theorem \ref{thm:dpsequivarbij}. 

\subsubsection{Resonance}
In \cite{DPS2015}, O.\ Pechenik and the first and second authors defined resonance to give language to describe a phenomenon observed when an action of interest (such as rowmotion) projects to a cyclic action of smaller order.

\begin{definition}
Suppose $G=\langle g \rangle$ is a cyclic group acting on a set $X$, $\mathcal{C}_{\omega}=\langle c \rangle$ a cyclic group of order $\omega$ acting nontrivially on a set $Y$, and $f:X \rightarrow Y$ a surjection. We say the triple $(X,G,f)$ exhibits \emph{resonance with frequency $\omega$} if, for all $x \in X$, $c \cdot f(x)=f(g \cdot x)$.
\end{definition}

Their prototypical example of resonance stems from increasing tableaux under $K$-promotion.

\begin{definition}
Define the binary content of an increasing tableau $T \in \inc{q}{\lambda}$ to be the sequence $\mathrm{Con}(T)=(a_1,a_2,\dots,a_q)$, where $a_i=1$ if $i$ is an entry of $T$ and $a_i=0$ if $i$ is not an entry of $T$.
\end{definition}

\begin{lemma}[\protect{\cite[Lemma 2.1]{DPS2015}}]
Let $T \in \inc{q}{\lambda}$. If $\mathrm{Con}(T)=(a_1,a_2,\dots,a_q)$, then $\mathrm{Con}(\kpro(T))$ $=(a_2,\dots,a_q,a_1)$.
\end{lemma}

In words, $K$-promotion on an increasing tableau cyclically shifts the binary content of the tableau. This gives the following statement of resonance, since though $K$-promotion on $\inc{q}{P}$ is often of order larger than $q$, this lemma shows it projects nicely to a cyclic action of order $q$.

\begin{theorem}[\protect{\cite[Theorem 2.2]{DPS2015}}]
$(\inc{q}{\lambda}, \langle \kpro \rangle, \mathrm{Con})$ exhibits resonance with frequency $q$.
\end{theorem}

Using Theorem~\ref{thm:dpsequivarbij}, they also obtained a resonance result for rowmotion.
We give analogous resonance results in the more general realm of increasing labelings in Corollaries~\ref{thm:res} and~\ref{cor:rowres}.

\section{Increasing labelings}
\label{sec:inc_labelings}
In this section, we extend the ideas behind the relationship between increasing tableaux of square shape and order ideals in the product of three chains poset to increasing labelings $\inc{R}{P}$ of any poset $P$ with restriction function $R$ and order ideals in the associated poset $\Gamma(P,R)$ of meet-irreducibles. In Subsection~\ref{subsec:gamma1}, we construct $\Gamma(P,R)$ and prove Theorem~\ref{thm:Gamma1MeetIrred} giving a bijection between $\inc{R}{P}$ and order ideals of $\Gamma(P,R)$. Next, in Subsection~\ref{subsec:q}, we restrict our attention to $\inc{q}{P}$ and its associated poset $\Gamma(P,q)$, which is simpler to describe than for general restriction functions.
Lastly, in Subsection~\ref{subsec:ranked} we consider how our framework naturally extended to the setting of weakly increasing labelings and prove an analogue of Theorem~\ref{thm:Gamma1MeetIrred} in this setting. We also discuss how weakly increasing labelings interact with strictly increasing labelings in the case where our poset is ranked.

\subsection{Construction of $\Gamma$}
\label{subsec:gamma1}
In this subsection, 
we construct $\Gamma(P,R)$, and prove Theorem~\ref{thm:Gamma1MeetIrred} showing the order ideals of $\Gamma(P,R)$ are in bijection with $\inc{R}{P}$, the increasing labelings of $P$ with ranges restricted by $R:P\mapsto \mathcal{P}(\mathbb{Z})$.

One natural restriction to place on $R$ is to require that if $k\in R(p)$, then there must be some increasing labeling $f\in\inc{R}{P}$ with $f(p)=k$. Otherwise, we could remove $k$ from the set of available labels for $p$ and not change the set  of allowable increasing labelings. We formalize this in the next proposition.

\begin{definition}
\label{def:consistent}
We say that a restriction function $R:P\mapsto \mathcal{P}(\mathbb{Z})$ is \emph{consistent} if for every covering relation $x\lessdot y$ in $P$, we have $\min(R(x)) < \min(R(y))$ and $\max(R(x))< \max(R(y))$.
\end{definition}

\begin{proposition}
\label{thm:inclabeliffconsistent}
Every $k\in R(p)$ has some increasing labeling $f\in\inc{R}{P}$ with $f(p)=k$ if and only if $R$ is consistent.
\end{proposition}

\begin{proof}
Assume $R$ is consistent. Pick an arbitrary $p\in P$ and $k\in R(p)$. Let $A$ be a maximal antichain in $P$ containing $p$. Let $B$ be the elements of $P$ strictly less than some element of $A$, and $C$ the set of elements in $P$ strictly greater than some element of $A$. By definition of $A$, every element of $P$ is in exactly one of $A$, $B$, or $C$. Define a function $f:P\mapsto \mathcal{P}(\mathbb{Z})$ by letting $f(p)=k$, $f(b)=\min(R(b))$ for $b\in B$, $f(c)=\max(R(c))$ for $c\in C$, and arbitrarily choosing $f(a)$ to be any element of $R(a)$. One can readily check from the definition of consistent that this defines an increasing labeling.

Conversely, assume $R$ is not consistent. Then there is some pair $p_1\lessdot p_2$ in $P$, with either $\min(R(p_1)) \geq \min(R(p_2))$ or $\max(R(p_1)) \geq \max(R(p_2))$. In the former case, there is clearly no increasing labeling where $f(p_2)=\min(R(p_2))$, and in the latter case there is clearly no labeling where $f(p_1)=\max(R(p_1))$.
\end{proof}

We now consider $\inc{R}{P}$ as a partially ordered set, where $f\leq g$ if and only if $f(p)\leq g(p)$ for all $p\in P$. Furthermore, it is a lattice, with meet given by $(f \wedge g)(p)=\min(f(p),g(p))$ and join given by $(f \vee g)(p)=\max(f(p),g(p))$. One may easily check that this lattice is distributive, so we may apply Birkhoff's Representation Theorem (also known as the fundamental theorem of finite distributive lattices), which says that every finite distributive lattice is isomorphic to the lattice of order ideals for some associated poset. 

\begin{remark}
\label{rem:meetvsjoin}
Typically with the fundamental theorem of finite distributive lattices, one represents the lattice as the lattice of order ideals on the induced subposet of its join irreducible elements. An equivalent formulation is to represent the lattice as the lattice of order filters on the induced subposet of its meet irreducible elements. We will then take the dual to obtain a representation of our lattice by order ideals. See \cite[Chapter 3]{Stanley1986} for further background/terminology.
\end{remark}

Now, we use the following definition and lemmas in Theorem~\ref{thm:meetirred} to describe the meet irreducible elements of $\inc{R}{P}$.

\begin{definition}
Say that an element $p\in P$ in an increasing labeling $f\in\inc{R}{P}$ is \emph{raisable} if there exists another increasing labeling $g\in\inc{R}{P}$ where $f(p)<g(p)$, and $f(p')=g(p')$ for all $p'\in P$, $p'\neq p$. 
\end{definition}

\begin{lemma}
If an increasing labeling has two raisable elements, then it is not meet irreducible.
\end{lemma}

\begin{proof}
Let $f\in\inc{R}{P}$ have raisable elements $p_1,p_2\in P$. Then by definition, there are increasing labelings $g_1,g_2\in\inc{R}{P}$ with $f(p_1)<g_1(p_1)$, $f(p_2)<g_2(p_2)$, $f(p')=g_1(p')$ for $p'\neq p_1$, and
$f(p'')=g_2(p'')$ for $p''\neq p_2$. Then clearly $f=g_1 \wedge g_2$.
\end{proof}

So meet irreducibles must have zero or one raisable elements.

\begin{lemma}
There is only one meet irreducible in $\inc{R}{P}$ with zero raisable elements.
\end{lemma}

\begin{proof}
 Let $I_1$ be the set of maximal elements of $P$, and recursively let $I_n$ be the maximal elements of $P\setminus \displaystyle\cup_{j=1}^{n-1} I_j$. Clearly the maximal elements have no restrictions above them limiting how large they can be, so if they are not raisable elements, then they all must be assigned the maximum possible label. Then recursively, we see that the only thing limiting $I_n$ is entries in $I_{n-1}$, but if all entries in $I_{n-1}$ are as large as possible, then all entries in $I_n$ must be as large as possible to keep from being raisable. Thus, the only increasing labeling with no raisable elements is the one given by $f(p)=\max(R(p))$. 
\end{proof}

\begin{definition}
Let $R(p)^*$ denote $R(p)$ with its largest element removed.
\end{definition}

\begin{lemma}
For every element $p\in P$ and $k\in R(p)^*$, there is exactly one meet irreducible in $\inc{R}{P}$ with $f(p)=k$ and $p$ as its only raisable element.
\end{lemma}

\begin{proof}
We will explicitly construct the only possible increasing labeling having $f(p)=k$ and $p$ as its only raisable element, and then show it is meet irreducible.

Let $I$ be the principle order ideal associated to $p$. Let $I_0=\{p\}$. Then recursively let $I_n$ be the maximal elements of $I\setminus\cup_{j=0}^{n-1} I_j$. In order for each $I_j$ to not have any raisable elements, each entry must be as large as possible while still respecting any relevant inequalities from $I_{j-1}$. So we may recursively define $f$ on each layer. For $p_1\in I_j$, we let $f(p_1)$ be the largest label in $R(p_1)$ that is less than $f(p_2)$ for every $p_1\lessdot p_2$ with $p_2$ in $I_{j-1}$. The restriction that $R$ is consistent ensures that such a largest label will always exist.

We may apply a similar logic to the elements of $P\setminus I$, except in this case there is no initial condition like $f(p)=k$, so every $p_1\in P\setminus I$ must have $f(p_1)=\max(R(p_1))$. The restriction that $R$ is consistent ensures that this will create an increasing labeling.

Now, we prove that this increasing labeling is meet irreducible. Assume that we have $f=g_1\wedge g_2$. Since $f$ takes the maximum possible value on elements incomparable to $p$, both $g_1$ and $g_2$ must also take the maximum value on these elements. Also, at least one of the two must map $p$ to $k$, so without loss of generality say $g_1(p)=k$. Then for every element less than $p$, $g_1$ must assign that element to something at least as big as $f$ does. But by construction of $f$, there is no larger possible assignment. Thus, $g_1=f$, and $f$ is meet irreducible.
\end{proof}

These lemmas yield the following proposition.

\begin{proposition}
\label{thm:meetirred}
The meet irreducibles of $\inc{R}{P}$ can be indexed by pairs $(p,k)$, where $p\in P$ and $k\in R(p)^*$.
\end{proposition}

See Figure~\ref{fig:gamma1tlexample} for an example.

If we make a poset of these meet irreducibles with the induced partial order on increasing labelings, then order filters in this poset will be in bijection with increasing labelings. Since we prefer to work with order ideals, we will then take the dual of this poset. This means we want to make a poset on pairs $(p,k)$, where $p\in P$ and $k\in R(p)^*$ where $(p_1,k_1)\leq (p_2,k_2)$ if and only if the meet irreducible corresponding to $(p_2,k_2)$ is less than or equal to the meet irreducible for $(p_1,k_1)$ in the natural ordering on $\inc{R}{P}$.
Since the meet irreducibles are recursively defined by making things beneath the indexed element as large as possible, it is more natural to describe the covering relations than the general partial order relations.

\begin{figure}[htbp]
\begin{tikzpicture}

\begin{scope}[xshift=0cm]
\coordinate (a) at (0,0);
\coordinate (b) at (-1,2);
\coordinate (c) at (1,1);
\coordinate (d) at (1,3);
\coordinate (e) at (0,4);
\draw[] (a) -- (b) -- (e) -- (d) -- (c) -- cycle ;
\node[fill=white,draw,circle,inner sep=.5ex] at (a) {a};
\node[fill=white,draw,circle,inner sep=.5ex] at (b) {b};
\node[fill=white,draw,circle,inner sep=.5ex] at (c) {c};
\node[fill=white,draw,circle,inner sep=.5ex] at (d) {d};
\node[fill=white,draw,circle,inner sep=.5ex] at (e) {e};
\node [below=.25cm] at (a) {1};
\node [left=.25cm] at (b) {3};
\node [right=.25cm] at (c) {5};
\node [right=.25cm] at (d) {6};
\node [above=.25cm] at (e) {9};
\end{scope}
\begin{scope}[xshift=5cm]
\coordinate (a) at (0,0);
\coordinate (b) at (-1,2);
\coordinate (c) at (1,1);
\coordinate (d) at (1,3);
\coordinate (e) at (0,4);
\draw[] (a) -- (b) -- (e) -- (d) -- (c) -- cycle ;
\node[fill=white,draw,circle,inner sep=.5ex] at (a) {a};
\node[fill=white,draw,circle,inner sep=.5ex] at (b) {b};
\node[fill=white,draw,circle,inner sep=.5ex] at (c) {c};
\node[fill=white,draw,circle,inner sep=.5ex] at (d) {d};
\node[fill=white,draw,circle,inner sep=.5ex] at (e) {e};
\node [below=.25cm] at (a) {1};
\node [left=.25cm] at (b) {5};
\node [right=.25cm] at (c) {4};
\node [right=.25cm] at (d) {5};
\node [above=.25cm] at (e) {9};
\end{scope}
\end{tikzpicture}
\caption{Meet irreducible increasing labelings indexed by $(b,3)$ (left) and by $(d,5)$ (right) for the poset and restriction function shown in the top of Figure~\ref{fig:TogIsProEx_intro}.}
\label{fig:gamma1tlexample}
\end{figure}
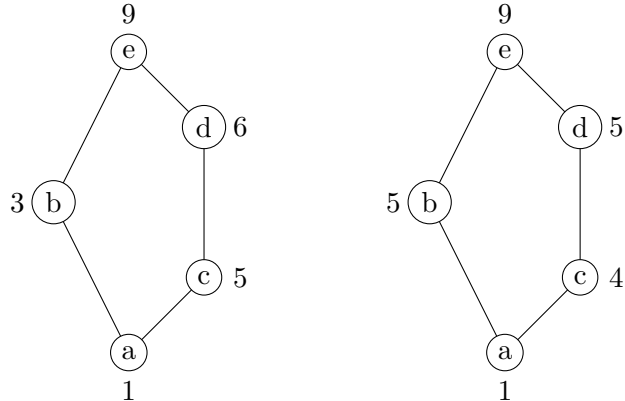

We will first introduce some convenient notation.
\begin{definition}
Let $R(p)_{>k}$ be the smallest label of $R(p)$ that is larger than $k$, and let $R(p)_{<k}$ be the largest label of $R(p)$ less than $k$.
\end{definition}

\begin{definition}
\label{def:GammaOne}
Let $P$ be a poset and $R$ a consistent map of possible labels. Then define $\Gamma(P,R)$ to be the poset whose elements are $(p,k)$ with $p\in P$ and $k\in R(p)^*$, and covering relations given by $(p_1,k_1)\lessdot (p_2,k_2)$ if and only if either

\begin{enumerate}
\item  $p_1=p_2$ and $R(p_1)_{>k_2}=k_1$ (i.e., $k_1$ is the next largest possible label after $k_2$), or

\item $p_1\lessdot p_2$ (in $P$), $k_1=R(p_1)_{<k_2}\neq \max(R(p_1))$, and no greater $k$ in $R(p_2)$ has $k_1=R(p_1)_{<k}$. That is to say, $k_1$ is the largest label of $R(p_1)$ less than $k_2$ ($k_1\neq \max(R(p_1)))$, and there is no greater $k\in R(p_2)$ having $k_1$ as the largest label of $R(p_1)$ less than $k$.
\end{enumerate}
\end{definition}

Regarding Remark~\ref{rem:meetvsjoin}, whether we proceeded with meet or join irreducibles, they would be both be indexed by pairs $(p,k)$, and the induced relations would give $(p_1,k_1)\leq (p_2,k_2)$ implies $p_2\leq p_1$. Since we prefer a representation where the order on $P$ is preserved instead of reversed, we would be taking a dual either way. Starting with order filters and dualizing naturally gives us order ideals, while starting with order ideals would require us to pass to the complementary order filter before dualizing.

\begin{remark}
In $\Gamma(P,R)$, we lose the information about $\max(R(p))$ for each $p\in P$. So when we draw $\Gamma(P,R)$, we add a label $(p,\max(R(p)))$ underneath the chain of elements of the form $(p,k)$ for $k\in R(p)^*$. This is a reminder that when an order ideal contains no elements of the form $(p,k)$, in the corresponding increasing labeling, the element $p$ is sent to $\max(R(p))$.
\end{remark}

\begin{example}
\label{ex:gamma1tlexample}
We show an example of the construction of $\Gamma(P,R)$ using the poset $P$ and restriction $R$ given by Figure \ref{fig:gamma1tlexample}, left. For each element of $P$, Condition (1) of Definition~\ref{def:GammaOne} gives us a vertical chain in $\Gamma(P,R)$. For example, $e$ has restricted entries $\{4,6,7,9\}$, which results in a chain in $\Gamma(P,R)$ with elements that correspond to 4, 6, and 7. As the largest restricted entry in $R(e)$, the 9 does not correspond to an element of $\Gamma(P,R)$. See Figure \ref{fig:gamma1tlexample}, right. Condition (2) gives us covering relations between these vertical chains. The first part of Condition (2) 
is essential for the bijection between order ideals of $\Gamma(P,R)$ and increasing labelings of $P$ with restriction function $R$. The second part of Condition (2)
is necessary in defining covering relations. For example, in Figure \ref{fig:gamma1tlexample}, consider $a$ and $b$ to be $p_1$ and $p_2$ respectively. If $k_2=2$, 1 is the largest label of $p_1$ less than $k_2$. However, because this also occurs when $k_2=3$ and $(p_2,3)\lessdot(p_2,2)$, we do not need a covering relation between $(p_1,1)$ and $(p_2,2)$ in $\Gamma(P,R)$.
 \end{example}

\begin{figure}[htbp]
\begin{tikzpicture}
\begin{scope}[xshift=0cm]
\coordinate (a) at (0,0);
\coordinate (b) at (-1,2);
\coordinate (c) at (1,1);
\coordinate (d) at (1,3);
\coordinate (e) at (0,4);
\draw[] (a) -- (b) -- (e) -- (d) -- (c) -- cycle ;
\node[fill=white,draw,circle,inner sep=.5ex] at (a) {a};
\node[fill=white,draw,circle,inner sep=.5ex] at (b) {b};
\node[fill=white,draw,circle,inner sep=.5ex] at (c) {c};
\node[fill=white,draw,circle,inner sep=.5ex] at (d) {d};
\node[fill=white,draw,circle,inner sep=.5ex] at (e) {e};
\node [below=.25cm] at (a) {\{1,4\}};
\node [left=.25cm] at (b) {\{2,3,5\}};
\node [right=.25cm] at (c) {\{2,4,5\}};
\node [right=.25cm] at (d) {\{3,4,5,6\}};
\node [above=.25cm] at (e) {\{4,6,7,9\}};
\end{scope}

\begin{scope}[xshift=8cm,scale=1.4]
\coordinate (a1) at (0,0);
\coordinate (a4) at (0,-.75);
\coordinate (b2) at (-2,2);
\coordinate (b3) at (-2,1);
\coordinate (b5) at (-2,0.25);
\coordinate (c2) at (2,2);
\coordinate (c4) at (2,1);
\coordinate (c5) at (2,0.25);
\coordinate (d3) at (1,4);
\coordinate (d4) at (1,3);
\coordinate (d5) at (1,2);
\coordinate (d6) at (1,1.25);
\coordinate (e4) at (0,5);
\coordinate (e6) at (0,4);
\coordinate (e7) at (0,3);
\coordinate (e9) at (0,2.25);
\draw[] (b2) -- (b3);
\draw[] (c2) -- (c4);
\draw[] (d3) -- (d4) -- (d5);
\draw[] (e4) -- (e6) -- (e7);
\draw[] (a1) -- (b3);
\draw[] (a1) -- (c4);
\draw[] (b3) -- (e4);
\draw[] (c4) -- (d5);
\draw[] (c2) -- (d4);
\draw[] (d5) -- (e6);
\draw[] (d3) -- (e4);
\node[fill=white,draw,circle,inner sep=.5ex] at (a1) {a,1};
\node[] at (a4) {a,4};
\node[fill=white,draw,circle,inner sep=.5ex] at (b2) {b,2};
\node[fill=white,draw,circle,inner sep=.5ex] at (b3) {b,3};
\node[] at (b5) {b,5};
\node[fill=white,draw,circle,inner sep=.5ex] at (c2) {c,2};
\node[fill=white,draw,circle,inner sep=.5ex] at (c4) {c,4};
\node[] at (c5) {c,5};
\node[fill=white,draw,circle,inner sep=.5ex] at (d3) {d,3};
\node[fill=white,draw,circle,inner sep=.5ex] at (d4) {d,4};
\node[fill=white,draw,circle,inner sep=.5ex] at (d5) {d,5};
\node[] at (d6) {d,6};
\node[fill=white,draw,circle,inner sep=.5ex] at (e4) {e,4};
\node[fill=white,draw,circle,inner sep=.5ex] at (e6) {e,6};
\node[fill=white,draw,circle,inner sep=.5ex] at (e7) {e,7};
\node[] at (e9) {e,9};
\end{scope}
\end{tikzpicture}
\caption{Left: A poset $P$ with restriction function $R$; Right: The corresponding poset $\Gamma(P,R)$.}
\label{fig:gamma1tlexample}
\end{figure}
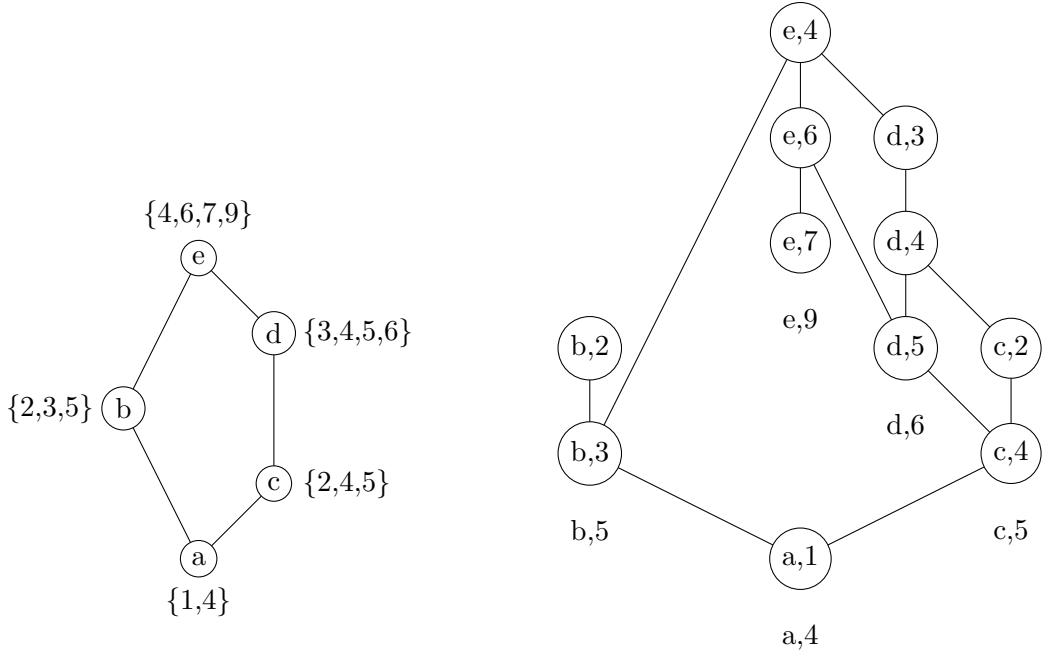

\begin{theorem}
\label{thm:Gamma1MeetIrred}
The poset $\Gamma(P,R)$ is isomorphic to the dual of the lattice of meet irreducibles of $\inc{R}{P}$. Therefore, order ideals of $\Gamma(P,R)$ are in bijection with $\inc{R}{P}$.
\end{theorem}

\begin{proof}
We will show that we have a covering relation of the meet irreducible indexed by $(p_2,k_2)$ being covered by the meet irreducible indexed by $(p_1,k_1)$ in the lattice of meet irreducibles of $\inc{R}{P}$ precisely when the conditions given for $(p_1,k_1)$ being covered by $(p_2,k_2)$ in $\Gamma(P,R)$ hold.

First, we show that the meet irreducible indexed by $(p_2,k_2)$ can only be covered by the meet irreducible indexed by $(p_1,k_1)$ if $p_1=p_2$, or $p_1\lessdot p_2$ in $P$.

Let $f\in\inc{R}{P}$ be the meet irreducible indexed by $(p_1,k_1)$ and $g$ the meet irreducible indexed by $(p_2,k_2)$. If $p_1$ and $p_2$ are incomparable, then $f(p_1)=k_1<\max(R(p_1))=f(p_2)$ and $g(p_2)=k_2<\max(R(p_2))=g(p_1)$, and thus the two meet irreducibles are incomparable. 

Similarly, if $p_1>p_2$ in $P$, then $g(p_2)=\max(R(p_2))>k_1=f(p_2)$, and thus the meet irreducible indexed by $(p_2,k_2)$ is not less than the meet irreducible indexed by $(p_1,k_1)$.

Now assume that $p_1<p_2$, but is not a covering relation in $P$. We know that $f(p_2)=\max(R(p_2))>k_2=g(p_2)$. If $g(p_1)>k_1$, then they are incomparable. If $g(p_1)\leq k_1 < \max(R(p_1))$, then there is some $p'$ with $p_1\lessdot p'$ and $g(p')<\max(R(p'))$. This is because the meet irreducible $g$ is defined so that everything is as large as it can possibly be for a valid increasing labeling, subject only to the condition that $g(p_2)=k_2$, and if all $p_1\lessdot p'$ had $g(p')=\max(R(p'))$, then $g(p_1)$ would be $\max(R(p_1))$. One can check that the meet irreducible indexed by $(p',g(p'))$ will be strictly between $f$ and $g$.

When $p_1=p_2=p$, then $g$ is less than or equal to $f$ if and only if $k_2\leq k_1$. The only thing we need to check is that no other meet irreducibles lie between them when there is no value of $k$ in $R(p)$ between $k_1$ and $k_2$. We only need to verify for meet irreducibles indexed by $(p',k)$ where either $p'\lessdot p$ or $p\lessdot p'$.

Assume there is no entry in $R(p)$ between $k_1$ and $k_2$. Assume that we had a meet irreducible $h$ indexed by $(p',k)$ such that $g<h<f$ in $\inc{R}{P}$. As there is no entry in $R(p)$ between $k_1$ and $k_2$, we must have $p'\neq p$. If $p'<p$, then $h(p)=\max(R(p))$ and $f(p)=k_1<\max(R(p))$, contradicting $h<f$. Similarly, if $p<p'$, then $g(p')=\max(R(p'))$, and $h(p')=k<\max(R(p'))$, contradicting $g<h$.

Lastly, we consider when $p_1\lessdot p_2$ in $P$. Then we can only fail to be a covering relation between $g$ and $f$ if there is a meet irreducible $h$ indexed by $(p,k)$ such that $g<h<f$. Without loss of generality, we can assume that $h\lessdot f$ in $\inc{R}{P}$. This means we must have $p_1=p$ or $p_1\lessdot p$. If $p_1=p$, then $h$ can exist if any only if there is a larger label of $R(p_1)$ that is still less than $k_2$. If $p_1\lessdot p$, then since $p<p_2$ and $p_1\lessdot p_2$, we must have $p=p_2$. Then $k$ would be a greater element of $R(p_2)$ having $k_1$ as the largest label in $R(p_1)$ less than $k$.



\end{proof}

\begin{remark}
Note that if $R(p)$ consists of a single element, then there is no element in $\Gamma(P,R)$ indexed by $p$. This is not an issue, because the only restrictions on an increasing labeling $f$ involving $p$ are that $f(p_1)<f(p)$ if $p_1\lessdot p$ and $f(p)<f(p_2)$ if $p\lessdot p_2$. The restriction that $R$ is consistent forces $f(p_1)\leq \max(R(p_1)) <\max(R(p))=f(p)$ and $f(p)=\min(R(p))<\min(R(p_2))\leq f(p_2)$.
\end{remark}

\begin{remark}
Presumably, one could extend this to the case where $P$ is infinite with sufficient care. Additionally, one could also extend this to the case where $R(p)$ is possibly unbounded. Consistency would require that things unbounded above be covered by things also unbounded above, and similarly that things unbounded below only cover things also unbounded below. Then the construction of $\Gamma(P,R)$ would have exactly the same definition, but its order ideals would be in bijection with increasing labelings of $P$ restricted by $R$ where entries unbounded above may be labeled $+\infty$ and things unbounded below may be labeled $-\infty$ (allowing $+\infty<+\infty$ and $-\infty<-\infty$).
\end{remark}

\subsection{Restricting the global set of labels}
\label{subsec:q}
One special case of interest is when the only restriction we place is that the  labels are in the bounded set $[q]=\{1,\ldots, q\}$. For example, in $K$-theory of the Grassmannian, increasing tableaux that only use the labels between $1$ and some fixed number $q$ are of interest \cite{BPS2016, BKSTY2008, DPS2015, Pechenik2017, TY2009, TY2011}. We denote this set $\inc{q}{P}$.

However, it is not as simple as setting $R(p)=[q]$ for every $p\in P$, as this is not a consistent set of possible labels. For example, minimal elements are the only elements of $P$ that can possibly be labeled $1$ (or else something that element covers would not have any potential labels), and maximal elements are the only ones that can possibly be labeled $q$ (or else something covering that element would not have any potential labels).

In general, the range of possible values for a particular element is determined by the length of a maximum length chain containing that element.

\begin{definition}
\label{def:MaxChainLen}
Following the notation of Stanley~\cite{Stanley1972}, for $p\in P$, let $\delta(p)$ be the number of elements less than $p$ in a maximum length chain containing $p$, and let $\nu(p)$ be the number of elements greater than $p$ in a maximum length chain containing $p$.
\end{definition}

\begin{lemma}
If every chain in $P$ has length at most $q$, then the map $R$ taking $p$ to $[1+\delta(p),q-\nu(p)]$ is consistent.
\end{lemma}

\begin{proof}
If $P$ has any chains of length greater than $q$, there is no way to label all these elements in a strictly increasing fashion with $q$ numbers, and $\inc{q}{P}$ would be empty. This is equivalent to there being some $p\in P$ with $\delta(p)+\nu(p)+1>q$, which implies the interval $[1+\delta(p),q-\nu(p)]$ is empty.

Otherwise, for an increasing labeling on $P$ with largest possible entry $q$, an element $p\in P$ can only be labeled by values on the interval $[1+\delta(p),q-\nu(p)]$. Additionally, if $p_1\lessdot p_2$, then $\delta(p_1)<\delta(p_2)$ (any maximum length chain of length $\ell$ with $p_1$ as its top element can be extended to a chain of length $\ell +1$ with $p_2$ as its top element), and $\nu(p_1)>\nu(p_2)$ (any maximum length chain of length $\ell$ with $p_2$ as its bottom element can be extended to a chain of length $\ell +1$ with $p_1$ as its bottom element). Thus, the map $R(p)=[1+\delta(p),q-\nu(p)]$ will be consistent.
\end{proof}

For the rest of this section, we will ignore the degenerate case, and assume all chains containing $p$ are of length less than $q$.

\begin{definition}
\label{def:GammaOneq}
Let $\Gamma(P,q)$ be the poset $\Gamma(P,R)$ for the restriction function given by $R(p)=[1+\delta(p),q-\nu(p)]$.
\end{definition}

\begin{example}
Consider $b$, the leftmost element of the poset $P$ in Figure \ref{fig:qexample}. The maximum chain containing $b$ contains itself and its cover, $e$. We see that $\delta(b)=0$, $\nu(b)=1$, and $q=5$, so the $b$ can be given any label in the interval $[1,4]$.
On the other hand, $e$ is contained in a maximum chain of length three. We obtain $\delta(e)=2$ and $\nu(e)=0$, making the restriction for $e$ the interval $[3,5]$. 
Applying this definition to every element in the poset results in a consistent labeling. The corresponding poset, $\Gamma(P,5)$, is shown in Figure \ref{fig:gamma1qexample}. See Figures~\ref{fig:inclabeling} and \ref{fig:orderideal} for an example of an increasing labeling in $\inc{5}{P}$ and its corresponding order ideal in $J(\Gamma(P,5))$.
\end{example}

We may obtain a simpler description of the covering relations in $\Gamma(P,q)$ than in the case of general ranges, because the range of each possible entry is an interval. (See Figure~\ref{fig:notranked})

\begin{theorem}
\label{thm:interval}
Let $R$ be a consistent restriction function for a poset $P$ such that $R(p)$ is always a non-empty interval. Then $\Gamma(P,R)$ is the poset with elements $\{(p,k) \ | \ p\in P \mbox{ and } k\in R(p)^* \}$ and covering relations given by $(p_1,k_1)\lessdot (p_2,k_2)$ if and only if either
\begin{enumerate}
\item
$p_1=p_2$ and $k_1=k_2+1$, or
\item
$p_1\lessdot_P p_2$ and $k_1+1=k_2$.
\end{enumerate}
\end{theorem}

\begin{proof}
We need to see how the covering relations in Definition~\ref{def:GammaOne} simplify when we know that our restriction function consists of intervals.

The first type of covering relation in $\Gamma(P,R)$ corresponds to the chain of elements for a fixed $p\in P$. Since the range of possible values for each $p$ is an interval, it is clear that the first type of covering relation in Definition~\ref{def:GammaOne} simplifies as above.

For the second type of covering relation, in Definition~\ref{def:GammaOne}, we need $k_1$ to be largest label of $R(p_1)$ less than $k_2$ ($k_1\neq \max(R(p_1)))$, and there can be no greater $k\in R(p_2)$ having $k_1$ as the largest label of $R(p_1)$ less than $k$.

First, we claim that if such a covering relation exists, then we necessarily have $k_1=k_2-1$. For assume we have a covering relation $(p_1,k_1)\lessdot (p_2,k_2)$ in $\Gamma(P,q)$ with $k_1<k_2-1$. By definition, $k_1$ is the largest label of $R(p_1)$ less than $k_2$, and $k_1$ is not $\max(R(p_1))$. But since $R(p_1)$ is an interval, and $k_1$ is not the maximum element, then $k_1+1$ would be a larger element of $R(p_1)$ less than $k_2$, which is a contradiction.

The second thing we need to show is that if we have $(p_1,k_2-1)$ and $(p_2,k_2)$ with $p_1\lessdot p_2$, $k_2\in R(p_1)^*$, and $k_2-1\in R(p_2)^*$, then these two elements form a covering relation in $\Gamma(P,R)$. This could only fail to be a covering relation if there was a greater $k\in R(p_2)$ having $k_2-1$ as the largest label of $R(p_1)$ less than $k$. However, since $k_2-1$ is in $R(p_1)^*$, it is not the largest element, so $k_2\in R(p_1)$. Thus, any $k\in R(p_2)$ greater that $k_2$ will have the largest element of $R(p_1)$ less than it be at least $k_2$.
\end{proof}

\begin{corollary}
\label{cor:GammaOneQ}
$\Gamma(P,q)$ is the poset with elements $\{(p,k) \ | \ p\in P \mbox{ and } k\in [1+\delta(p),q-\nu(p)-1]\}$, and covering relations given by $(p_1,k_1)\lessdot (p_2,k_2)$ if and only if either

\begin{enumerate}
\item
$p_1=p_2$ and $k_1=k_2+1$, or
\item
$p_1\lessdot_P p_2$ and $k_1+1=k_2$.
\end{enumerate}
\end{corollary}

\begin{figure}[htbp]
\begin{tikzpicture}
\coordinate (a) at (0,0);
\coordinate (d) at (1,1);
\coordinate (c) at (-1,1);
\coordinate (e) at (-2,2);
\coordinate (b) at (-3,1);
\draw[] (d) -- (a) -- (c) -- (e) -- (b) ;
\node[fill=white,draw,circle,inner sep=.5ex] at (a) {a};
\node[fill=white,draw,circle,inner sep=.5ex] at (d) {d};
\node[fill=white,draw,circle,inner sep=.5ex] at (c) {c};
\node[fill=white,draw,circle,inner sep=.5ex] at (e) {e};
\node[fill=white,draw,circle,inner sep=.5ex] at (b) {b};
\node[below=.25cm] at (a) {$\{1,2,3\}$};
\node[above=.25cm] at (d) {$\{2,3,4,5\}$};
\node[above=.25cm] at (c) {\hspace{.4cm}$\{2,3,4\}$};
\node[above=.25cm] at (e) {$\{3,4,5\}$};
\node[below=.25cm] at (b) {$\{1,2,3,4\}$};
\end{tikzpicture}
\caption{A poset $P$ with the ranges of possible values for each entry in an increasing labeling in $\inc{5}{P}$.}
\label{fig:qexample}
\end{figure}
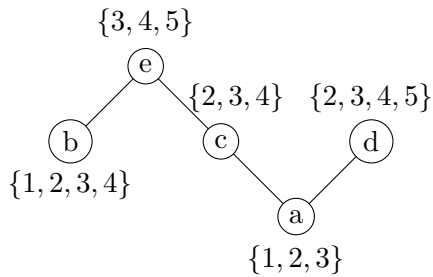

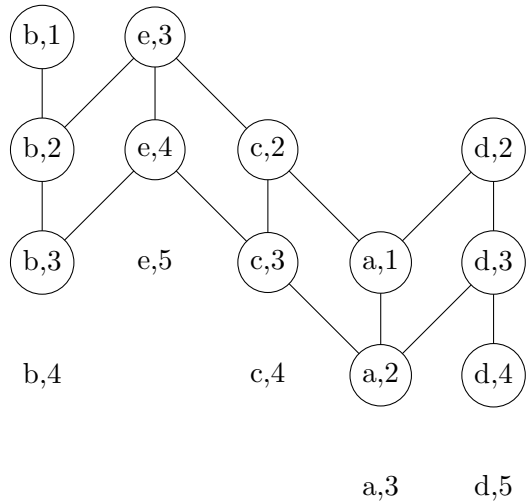
\begin{figure}[htbp]
\begin{tikzpicture}[scale=1.5]
\coordinate (a0) at (0,-1);
\coordinate (a) at (0,0);
\coordinate (a1) at (0,1);
\coordinate (d-1) at (1,-1);
\coordinate (d0) at (1,0);
\coordinate (d) at (1,1);
\coordinate (d1) at (1,2);
\coordinate (c0) at (-1,0);
\coordinate (c) at (-1,1);
\coordinate (c1) at (-1,2);
\coordinate (e0) at (-2,1);
\coordinate (e) at (-2,2);
\coordinate (e1) at (-2,3);
\coordinate (b0) at (-3,0);
\coordinate (b) at (-3,1);
\coordinate (b1) at (-3,2);
\coordinate (b2) at (-3,3);
\draw[] (d) -- (a) -- (c) -- (e) -- (b) ;
\draw[] (d1) -- (a1) -- (c1) -- (e1) -- (b1) ;
\draw (a1) -- (a);
\draw (d1) -- (d);
\draw (c1) -- (c);
\draw (e1) -- (e);
\draw (b1) -- (b);
\draw (d0) -- (d);
\draw (b2) -- (b1);
\node[inner sep=.5ex] at (a0) {a,3};
\node[fill=white,draw,circle,inner sep=.5ex] at (a) {a,2};
\node[fill=white,draw,circle,inner sep=.5ex] at (a1) {a,1};
\node[inner sep=.5ex] at (d-1) {d,5};
\node[fill=white,draw,circle,inner sep=.5ex] at (d0) {d,4};
\node[fill=white,draw,circle,inner sep=.5ex] at (d) {d,3};
\node[fill=white,draw,circle,inner sep=.5ex] at (d1) {d,2};
\node[inner sep=.5ex] at (c0) {c,4};
\node[fill=white,draw,circle,inner sep=.5ex] at (c) {c,3};
\node[fill=white,draw,circle,inner sep=.5ex] at (c1) {c,2};
\node[inner sep=.5ex] at (e0) {e,5};
\node[fill=white,draw,circle,inner sep=.5ex] at (e) {e,4};
\node[fill=white,draw,circle,inner sep=.5ex] at (e1) {e,3};
\node[inner sep=.5ex] at (b0) {b,4};
\node[fill=white,draw,circle,inner sep=.5ex] at (b) {b,3};
\node[fill=white,draw,circle,inner sep=.5ex] at (b1) {b,2};
\node[fill=white,draw,circle,inner sep=.5ex] at (b2) {b,1};
\end{tikzpicture}
\caption{$\Gamma(P,5)$, the poset whose order ideals are in bijection with $\inc{5}{P}$.}
\label{fig:gamma1qexample}
\end{figure}

\begin{figure}[htbp]
\begin{tikzpicture}
\coordinate (a) at (0,0);
\coordinate (d) at (1,1);
\coordinate (c) at (-1,1);
\coordinate (e) at (-2,2);
\coordinate (b) at (-3,1);
\draw[] (d) -- (a) -- (c) -- (e) -- (b) ;
\node[fill=white,draw,circle,inner sep=.5ex] at (a) {a};
\node[fill=white,draw,circle,inner sep=.5ex] at (d) {d};
\node[fill=white,draw,circle,inner sep=.5ex] at (c) {c};
\node[fill=white,draw,circle,inner sep=.5ex] at (e) {e};
\node[fill=white,draw,circle,inner sep=.5ex] at (b) {b};
\node[below=.25cm] at (a) {1};
\node[above=.25cm] at (d) {3};
\node[right=.25cm] at (c) {3};
\node[above=.25cm] at (e) {5};
\node[below=.25cm] at (b) {2};
\end{tikzpicture}
\caption{An increasing labeling on $P$ with labels at most 5.}
\label{fig:inclabeling}
\end{figure}
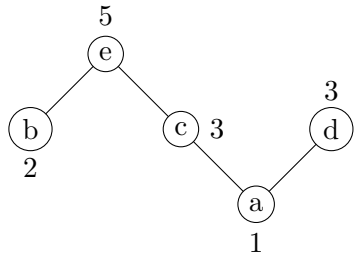

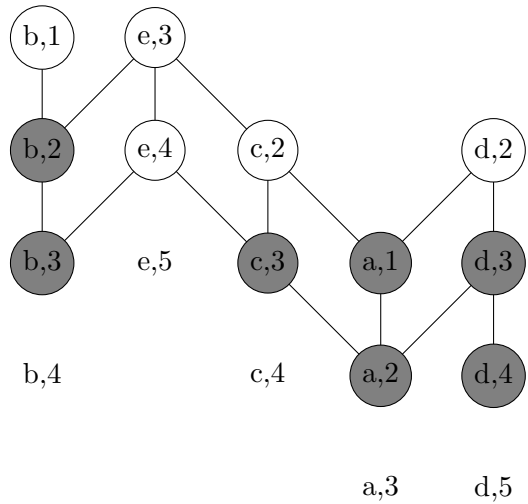
\begin{figure}[htbp]
\begin{tikzpicture}[scale=1.5]
\coordinate (a0) at (0,-1);
\coordinate (a) at (0,0);
\coordinate (a1) at (0,1);
\coordinate (d-1) at (1,-1);
\coordinate (d0) at (1,0);
\coordinate (d) at (1,1);
\coordinate (d1) at (1,2);
\coordinate (c0) at (-1,0);
\coordinate (c) at (-1,1);
\coordinate (c1) at (-1,2);
\coordinate (e0) at (-2,1);
\coordinate (e) at (-2,2);
\coordinate (e1) at (-2,3);
\coordinate (b0) at (-3,0);
\coordinate (b) at (-3,1);
\coordinate (b1) at (-3,2);
\coordinate (b2) at (-3,3);
\draw[] (d) -- (a) -- (c) -- (e) -- (b) ;
\draw[] (d1) -- (a1) -- (c1) -- (e1) -- (b1) ;
\draw (a1) -- (a);
\draw (d1) -- (d);
\draw (c1) -- (c);
\draw (e1) -- (e);
\draw (b1) -- (b);
\draw (d0) -- (d);
\draw (b2) -- (b1);
\node[inner sep=.5ex] at (a0) {a,3};
\node[fill=gray,draw,circle,inner sep=.5ex] at (a) {a,2};
\node[fill=gray,draw,circle,inner sep=.5ex] at (a1) {a,1};
\node[inner sep=.5ex] at (d-1) {d,5};
\node[fill=gray,draw,circle,inner sep=.5ex] at (d0) {d,4};
\node[fill=gray,draw,circle,inner sep=.5ex] at (d) {d,3};
\node[fill=white,draw,circle,inner sep=.5ex] at (d1) {d,2};
\node[inner sep=.5ex] at (c0) {c,4};
\node[fill=gray,draw,circle,inner sep=.5ex] at (c) {c,3};
\node[fill=white,draw,circle,inner sep=.5ex] at (c1) {c,2};
\node[inner sep=.5ex] at (e0) {e,5};
\node[fill=white,draw,circle,inner sep=.5ex] at (e) {e,4};
\node[fill=white,draw,circle,inner sep=.5ex] at (e1) {e,3};
\node[inner sep=.5ex] at (b0) {b,4};
\node[fill=gray,draw,circle,inner sep=.5ex] at (b) {b,3};
\node[fill=gray,draw,circle,inner sep=.5ex] at (b1) {b,2};
\node[fill=white,draw,circle,inner sep=.5ex] at (b2) {b,1};
\end{tikzpicture}
\caption{The  order ideal in $\Gamma(P,5)$ corresponding to the increasing labeling in Figure~\ref{fig:inclabeling}.}
\label{fig:orderideal}
\end{figure}

\begin{figure}[htbp]
\label{fig:unranked}
\begin{tikzpicture}
\begin{scope}[xshift=0cm]
\coordinate (a) at (0,0);
\coordinate (b) at (-1,2);
\coordinate (c) at (1,1);
\coordinate (d) at (1,3);
\coordinate (e) at (0,4);
\draw[] (a) -- (b) -- (e) -- (d) -- (c) -- cycle ;
\node[fill=white,draw,circle,inner sep=.5ex] at (a) {a};
\node[fill=white,draw,circle,inner sep=.5ex] at (b) {b};
\node[fill=white,draw,circle,inner sep=.5ex] at (c) {c};
\node[fill=white,draw,circle,inner sep=.5ex] at (d) {d};
\node[fill=white,draw,circle,inner sep=.5ex] at (e) {e};
\end{scope}

\begin{scope}[xshift=8cm,scale=1.4]
\coordinate (a1) at (0,0);
\coordinate (a2) at (0,-1);
\coordinate (a3) at (0,-1.75);
\coordinate (b1) at (-2,3);
\coordinate (b2) at (-2,2);
\coordinate (b3) at (-2,1);
\coordinate (b4) at (-2,0);
\coordinate (b5) at (-2,-.75);
\coordinate (c2) at (2,2);
\coordinate (c3) at (2,1);
\coordinate (c4) at (2,0.25);
\coordinate (d3) at (1,4);
\coordinate (d4) at (1,3);
\coordinate (d5) at (1,2.25);
\coordinate (e4) at (0,5);
\coordinate (e5) at (0,4);
\coordinate (e6) at (0,3);
\draw[] (a2) -- (c3);
\draw[] (a1) -- (c2);
\draw[] (a2) -- (b3);
\draw[] (a1) -- (b2);
\draw[] (b4) -- (e5);
\draw[] (b3) -- (e4);
\draw[] (c3) -- (d4);
\draw[] (c2) -- (d3);
\draw[] (d4) -- (e5);
\draw[] (d3) -- (e4);
\draw[] (a1) -- (a2);
\draw[] (b2) -- (b3) -- (b4);
\draw[] (c2) -- (c3);
\draw[] (d3) -- (d4);
\draw[] (e4) -- (e5);
\node[fill=white,draw,circle,inner sep=.5ex] at (a1) {a,1};
\node[fill=white,draw,circle,inner sep=.5ex] at (a2) {a,2};
\node[] at (a3) {a,3};
\node[fill=white,draw,circle,inner sep=.5ex] at (b2) {b,2};
\node[fill=white,draw,circle,inner sep=.5ex] at (b3) {b,3};
\node[fill=white,draw,circle,inner sep=.5ex] at (b4) {b,4};
\node[] at (b5) {b,5};
\node[fill=white,draw,circle,inner sep=.5ex] at (c2) {c,2};
\node[fill=white,draw,circle,inner sep=.5ex] at (c3) {c,3};
\node[] at (c4) {c,4};
\node[fill=white,draw,circle,inner sep=.5ex] at (d3) {d,3};
\node[fill=white,draw,circle,inner sep=.5ex] at (d4) {d,4};
\node[] at (d5) {d,5};
\node[fill=white,draw,circle,inner sep=.5ex] at (e4) {e,4};
\node[fill=white,draw,circle,inner sep=.5ex] at (e5) {e,5};
\node[] at (e6) {e,6};
\end{scope}
\end{tikzpicture}
\caption{A poset $P$ that is not ranked and $\Gamma(P,6)$.}
\label{fig:notranked}
\end{figure}
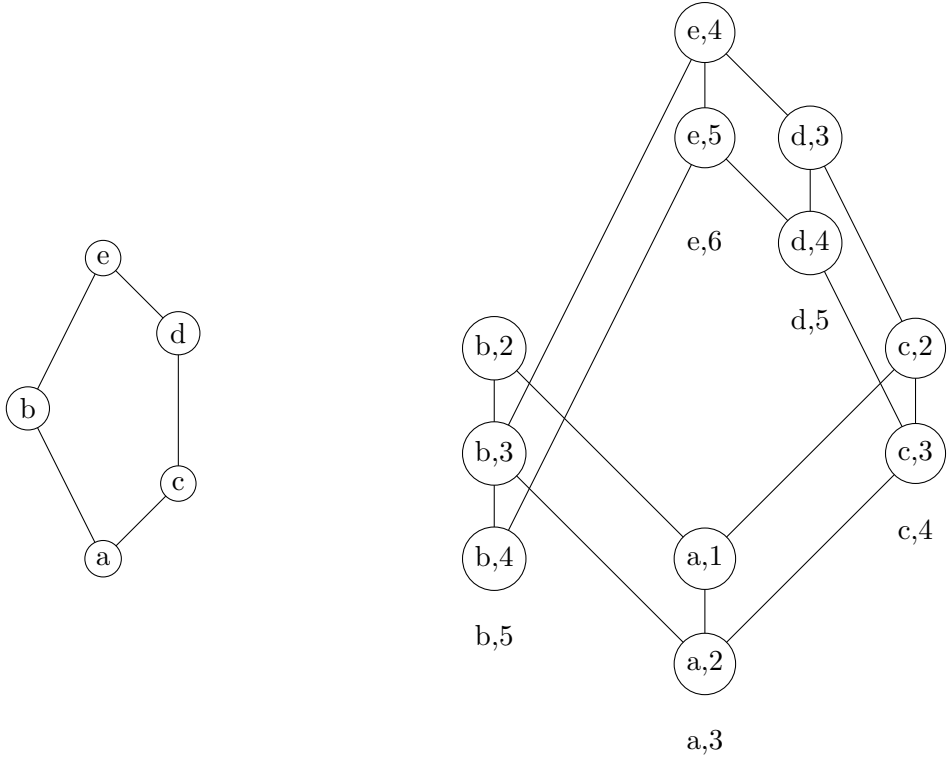

\subsection{Weakly increasing labelings}
\label{subsec:ranked}
With minor modification, everything that has been said so far about strictly increasing labelings also applies to weakly increasing labelings.

One can say that a restriction function $R:P\mapsto \mathcal{P}(\mathbb{Z})$ is \emph{weakly-consistent} if for every covering relation $x\lessdot y$ in $P$, we have $\min(R(x)) \leq \min(R(y))$ and $\max(R(x))\leq \max(R(y))$; this is Definition~\ref{def:consistent} with weak inequalities instead of strict. Then the same argument as in the proof of Theorem~\ref{thm:inclabeliffconsistent} shows that every potential label is used in some weakly increasing labeling if and only if the restriction function is weakly consistent. We denote weakly increasing labelings on $P$ with entries restricted by $R$ as $\winc{R}{P}$. 

The set of weakly increasing labelings is again partially ordered by component-wise comparison, with join and meet given by taking maxima and minima componentwise (respectively), and meet irreducibles indexed by $(p,k)$ where $k\neq \max(R(p))$.

\begin{definition}
Given a consistent restriction function $R$, let $R(p)_{\geq k}$ be the smallest label of $R(p)$ that is larger than or equal to $k$, and let $R(p)_{\leq k}$ be the largest label of $R(p)$ less than or equal to k.
\end{definition}

\begin{definition}
\label{def:GammaTwo}
Let $P$ be a poset and $R$ a weakly-consistent map of possible labels. Then define $\Gamma'(P,R)$ to be the poset whose elements are $(p,k)$ with $p\in P$ and $k\in R(p)^*$, and covering relations given by $(p_1,k_1)\lessdot (p_2,k_2)$ if and only if either

\begin{enumerate}
\item  $p_1=p_2$ and $R(p)_{>k_2}=k_1$ , or

\item $p_1\lessdot p_2$ (in $P$), $R(p_1)_{\leq k_2}=k_1$, and no greater $k\in R(p_2)$ has $R(p_1)_{\leq k}=k_1$
\end{enumerate}
\end{definition}

\begin{theorem}
The poset $\Gamma'(P,R)$ is isomorphic to the dual of the lattice of meet irreducibles of $\winc{R}{P}$. Therefore, order ideals of $\Gamma'(P,R)$ are in bijection with $\winc{R}{P}$.
\end{theorem}
There is a simple connection between strictly increasing labelings and weakly increasing labelings in the case when $P$ is ranked. We define a poset to be \emph{ranked} if there exists a rank function $\rk:P\rightarrow \mathbb{Z}$ such that $p_1\lessdot_P p_2$ implies $\rk(p_2)=\rk(p_1)+1$.

\begin{theorem}
If $P$ is a ranked poset with rank function $\rk$, then $R_1:P\mapsto \mathcal{P}(\mathbb{Z})$ is a consistent restriction function if and only if $R_2:P\mapsto \mathcal{P}(\mathbb{Z})$ given by $R_1(p)=R_2(p)+\rk(p)$ is a weakly-consistent labeling. Also, $\Gamma(P,R_1)\equiv\Gamma'(P,R_2)$.
\end{theorem}

The proof of this theorem consists of using the rank function to convert strict inequalities to weak inequalities.

In the case of strictly increasing labelings with global bound $q$ on a ranked poset, the description when we convert to weakly increasing labelings becomes even nicer. In particular, the covering relations on the ordered pairs $(p,k)$ appearing in $\Gamma'(P,q)$ are exactly the same as the induced subposet of the Cartesian product $P\times\mathbb{Z}^{rev}$. There is a natural bijection between strictly increasing labelings with entries in $[q]$ and strictly decreasing labelings with entries in $q$ (replace each label $i$ with $q+1-i$) which we exploit in order to get the following Cartesian product representation with the usual ordering on $\mathbb{Z}$.

\begin{theorem}
$\Gamma(P,q)$ is isomorphic to the induced subposet of $P\times \mathbb{Z}$ on the elements $\{(p,k)\in P\times[\mathbf{q}] \ | \ k\in [\nu(p)+\rk(p)+1,q-\delta(p)+\rk(p)-1] \}$
via the map $(p,k)\mapsto (p,q+1-k-\rk(p))$.
\end{theorem}

An additional special case of interest is when $P$ is ranked and the maximum length of a chain containing an element is the same for every element in $P$ (which Stanley calls the \emph{$\lambda$ chain condition}). In this case, $\delta(p)+1+\nu(p)$ is the same for every element $p$ (call this common number $\rk(P)$). One can check that  $\delta(p)$ is a rank function for this poset, and then $[\nu(p)+\rk(p)+1,q-\delta(p)+\rk(p)-1]=[\rk(P),q-1]$, which is independent of $p$. If we subtract $\rk(P)-1$ from both components (so our interval goes from 1 to $q-\rk(P)$ instead of from $\rk(P)$ to $q-1$) we obtain the following corollary.

\begin{corollary}
\label{cor:RankedIsRectangle}
Let $P$ be a ranked poset where the maximum length of a chain containing an element is the same for every element of $P$. Then $\Gamma(P,q)$ is in bijection with $P\times[{q-\mathrm{rk}(P)}]$.
\end{corollary}

\begin{remark}
The previous work of \cite{DPS2015} was largely concerned with the special case of this corollary. When $P$ is the product of chains poset $[a] \times [b]$, increasing labelings are increasing tableaux. Furthermore, $\Gamma(P,q)$ is in bijection with $\inc{q}{P}$ by Theorem \ref{thm:Gamma1MeetIrred}, and also a product of three chains poset by the previous result, Corollary \ref{cor:RankedIsRectangle}. Composing these bijections yields the bijection between the product of chains poset and increasing tableaux used in \cite{DPS2015} to prove Theorem \ref{thm:dpsequivarbij}.
\end{remark}

\begin{remark}
\label{rem:Ppart}
The connection between weakly increasing labelings being in bijection with order ideals in a Cartesian product goes back to Birkhoff~\cite{Birkhoff1940} (in the equivalent context of \emph{isotone functions}). The connection between strictly and weakly increasing labelings was looked at in the equivalent context of \emph{P-partitions} by Stanley~ \cite{Stanley1972}. Stanley primarily considers increasing labelings in the case where they can be transformed to weakly increasing labelings, and primarily the case with restricted largest possible entry. We consider increasing labelings of all posets, with arbitrary restricted parts, and work out the structure of the poset of join irreducibles in this broader context. Stanley derived the poset theoretic conditions under which the injective map $f(p)\mapsto f(p)+\delta(x)$ from all weakly increasing labelings of $P$ to all strictly increasing labelings is bijective. We are primarily interested in this relationship only when it helps makes the structure of $\Gamma(P,R)$ as a Cartesian product more transparent.
\end{remark}

\begin{remark}
\label{rem:Ppart2}
In the language of $P$-partitions, our weakly and strictly increasing labelings correspond to $(P,\omega)$ partitions where $\omega$ is a natural (resp. dual natural) labeling. Our results could be extended to labelings where we arbitrarily choose which inequalities are weak vs strict (corresponding to arbitrary $(P,\omega)$ partitions). A particular application of interest would be to semistandard Young tableaux, where the inequalities are weak along rows and strict along columns.
\end{remark}

\section{Promotion and resonance on increasing labelings}
\label{sec:Promotion}
In this section, we generalize jeu de taquin promotion to $\inc{q}{P}$ and show in
Theorem~\ref{thm:bk=jdt} this action is equivalent to increasing labeling promotion $\incpro$, defined in terms of generalized Bender-Knuth involutions. 
We then give a resonance result on $\incpro$, as a corollary of Theorem~\ref{thm:bk=jdt}.

We begin by noting how Definition~\ref{def:bkR} simplifies in the case of $\inc{q}{P}$.
\begin{proposition}
\label{def:bk}
For each $i\in\mathbb{Z}$,
 the $i$th Bender-Knuth involution $\rho_i:\inc{q}{P}\rightarrow \inc{q}{P}$ acts as follows. For $x\in P$,

\[\rho_i(f)(x)=\begin{cases}
i+1 &  f(x)=i \mbox{ and the resulting labeling is still in } \inc{q}{P}\\
i &  f(x)=i+1 \mbox{ and the resulting labeling is still in } \inc{q}{P}\\
f(x) & \mbox{otherwise.}
\end{cases}\]

That is, $\rho_i$ increments $i$ and/or decrements $i+1$ \emph{wherever possible}. Then increasing labeling promotion on $f$ is the product $\incpro(f)=\rho_{q-1}\circ\cdots\circ\rho_{3}\circ\rho_{2}\circ\rho_{1}(f)$.
\end{proposition}
\begin{proof}
The induced restriction function in $\inc{q}{P}$ is $R(x)=[1+\delta(x),q-\nu(x)]$ for all $x\in P$, where $\delta$ and $\nu$ are as in Definition~\ref{def:MaxChainLen}. Since $R(x)$ is always an interval, $R(x)_{>i}=i+1$ for all $i\neq q-\nu(x)$, and Definition~\ref{def:bkR} reduces to the above.
\end{proof}

\begin{remark}
If $P$ is a partition shaped poset, $\incpro$ on $\inc{q}{P}$ is $K$-promotion on increasing tableaux of that partition shape with entries at most $q$.
If we restrict to bijective labelings of an arbitrary poset $P$, this action reduces to  promotion on linear extensions of $P$. 
\end{remark}

We give another definition in terms of generalized jeu de taquin slides. 
\begin{definition}
\label{def:jdt}
Let $\mathbb{Z}_{\boxempty}(P)$ denote the set of labelings $g:P\rightarrow(\mathbb{Z}\cup\boxempty)$.
Define the $i$th \emph{jeu de taquin slide} $\sigma_i:\mathbb{Z}_{\boxempty}(P)\rightarrow \mathbb{Z}_{\boxempty}(P)$ 
as follows: 

\[\sigma_i(g)(x)=\begin{cases}
i &  g(x)=\boxempty \mbox{ and } g(y)=i \mbox{ for some } y\gtrdot x \\
\boxempty &  g(x)=i \mbox{ and } g(z)=\boxempty \mbox{ for some } z\lessdot x\\
g(x) & \mbox{otherwise.}
\end{cases}\]
In words, $\sigma_i(g)(x)$ replaces a label $\boxempty$ with $i$ if $i$ is the label of a cover of $x$, replaces a label $i$ by $\boxempty$ if $x$ covers an element labeled by $\boxempty$, and leaves all other labels unchanged.

Let $\sigma_{i\rightarrow j}:\mathbb{Z}_{\boxempty}(P)\rightarrow \mathbb{Z}_{\boxempty}(P)$ be defined as \[\sigma_{i\rightarrow j}(g)(x)=\begin{cases} j & g(x)=i\\  g(x) &\mbox{otherwise}.\end{cases}\]
In words, $\sigma_{i\rightarrow j}(g)(x)$ replaces all labels $i$ by $j$. 

For $f\in\inc{q}{P}$, let $\jdt(f)=\sigma_{\boxempty\rightarrow (q+1)}\sigma_{q}\circ\sigma_{q-1}\circ\cdots\circ\sigma_{3}\circ\sigma_2\circ\sigma_{1\rightarrow \boxempty}(f)$. That is, first replace all $1$ labels by $\boxempty$. Then perform the $i$th jeu de taquin slide $\sigma_i$ for all $2\leq i\leq q$. Next, replace all labels $\boxempty$ by $q+1$. 
Define \emph{jeu de taquin promotion} on $f$ as $\jdtpro(f)(x)=\jdt(f)(x)-1$.
\end{definition}

\begin{proposition}
For $f\in\inc{q}{P}$, $\jdtpro(f)\in\inc{q}{P}$.
\end{proposition}
\begin{proof}
$\jdtpro(f)$ is a labeling of $P$ by integers in $[q]$, by construction. We need only show the labeling 
preserves strict order relations. 
Since we are applying the jeu de taquin slides $\sigma_i$ in order $i=2,3,\ldots,q$, each $\boxempty$ will be filled by the smallest of the labels of its covers. 
\end{proof}

\begin{remark}
Note that the above proposition would not hold if we used $\inc{R}{P}$ instead of $\inc{q}{P}$, since the result of the generalized jeu de taquin slides and then subtracting one would no longer be guaranteed to produce labels in the ranges required by $R$.
\end{remark}

The next theorem shows that $\incpro(f)=\jdtpro(f)$ for all $f\in\inc{q}{P}$. We will need the following definition.

\begin{definition}
Define the \emph{sliding subposet} of $f$ as the induced subposet  $S(f) = \{x\in P \ | \ x$  takes the label  $\boxempty$ at some point when applying $\jdtpro \mbox{ to } f\}$.
\end{definition}

\begin{remark}
The sliding subposet coincides with the \emph{flow paths} of Pechenik in the case of increasing tableaux \cite{Pechenik2014} and with the \emph{jeu de taquin sliding path} or \emph{promotion path} in the case of standard Young tableaux \cite{Stanley2009}.
\end{remark}

\begin{theorem}
\label{thm:bk=jdt}
For $f\in\inc{q}{P}$, Bender-Knuth promotion on $f$ equals jeu de taquin promotion on $f$, that is, $\incpro(f)=\jdtpro(f)$.
\end{theorem}

\begin{proof}
Let $f\in\inc{q}{P}$. We wish to show that for all $x\in P$, 
\begin{equation}
\label{eq:jdtrho}
\rho_{q-1}\circ\rho_{q-2}\circ\cdots\circ\rho_2\circ\rho_1(f)(x)=\jdt(f)(x)-1.
\end{equation}

Note the sliding subposet $S(f)$ is a  subposet of $P$ whose minimal elements are minimal in $P$ and whose maximal elements are maximal in $P$, that is, $S(f)$ is a union of maximal saturated chains. This is because the $\boxempty$ labels are first introduced as labels of minimal elements, then at the intermediate steps $\sigma_i$, anywhere a $\boxempty$ label is replaced, other $\boxempty$ label(s) are introduced in a cover(s) of the previous element(s) labeled by $\boxempty$.

\smallskip
We prove Equation \eqref{eq:jdtrho} using three cases:

\smallskip
\textbf{Case $x\in S(f)$ and $x$ is not maximal in $P$:} In this case, $\jdt(f)(x)=f(y)$ where $y$ is such that $f(y)=\min\left\{f(z) \ | \ z\gtrdot x \right\}$. 
We now need to show $\rho_{q-1}\circ\rho_{q-2}\circ\cdots\circ\rho_2\circ\rho_1(f)(x)=f(y)-1$.

Let $e_0,e_1,\ldots,e_{k-1},e_k=x,e_{k+1},\ldots,e_{m}$ be a saturated chain in the sliding subposet $S(f)$ from a minimal element to a maximal element. It must be that each $e_i$ is such that $f(e_i)$ is the minimal label among all labels of covers of $e_{i-1}$, or else $e_i$ would not be on $S(f)$. We know $f(e_0)=1$ since $e_0$ is a minimal element of $P$ in $S(f)$. The Bender-Knuth involutions increment this label until it can no longer be incremented, that is, until it is one less than the smallest label of one (or more) of its covers. A cover with this smallest label is $e_1$, or else $e_1$ would not be on the sliding subposet. So $\rho_{f(e_1)-1}\circ\cdots\circ\rho_2\circ\rho_1(f)(e_0)=f(e_1)-1$.

Now suppose that $\rho_{f(e_i)-1}\circ\cdots\circ\rho_2\circ\rho_1(f)(e_{i-1})=f(e_i)-1$. We wish to show $\rho_{f(e_{i+1})-1}\circ\cdots\circ\rho_2\circ\rho_1(f)(e_{i})=f(e_{i+1})-1$. We have $f(e_{i})>f(e_{i-1})$ since $f$ is an increasing labeling, so we know $\rho_{f(e_{i-1})-1}\circ\cdots\circ\rho_2\circ\rho_1(f)(e_{i})=f(e_i)$. Then since $e_i\gtrdot e_{i-1}$ and $\rho_{f(e_i)-1}\circ\cdots\circ\rho_2\circ\rho_1(f)(e_{i-1})=f(e_i)-1$, $\rho_{f(e_i)-1}$ cannot decrement $f(e_i)$, since it covers an element with label $f(e_i)-1$. So $\rho_{f(e_i)-1}\circ\cdots\circ\rho_2\circ\rho_1(f)(e_{i})=f(e_i)$ and then the subsequent Bender-Knuth involutions increment the label of $e_i$ until it can no longer be incremented, that is, until the label on $e_i$ equals $f(e_{i+1})-1$. This proves the statement.

\smallskip
\textbf{Case $x\in S(f)$ and $x$ is maximal in $P$:} In this case, since $x$ is on the sliding subposet, it must be that $\sigma_{q}\circ\sigma_{q-1}\circ\cdots\circ\sigma_{3}\circ\sigma_2\circ\sigma_{1\rightarrow \boxempty}(f)=\boxempty$, so $\sigma_{\boxempty\rightarrow q+1}$ fills $\boxempty$ with $q+1$. Thus  $\jdt(f)(x)=q+1$, so $\jdt(f)(x)-1=q$. Thus we need to show $\rho_{q-1}\circ\rho_{q-2}\circ\cdots\circ\rho_2\circ\rho_1(f)(x)=q$ as well. 
The proof follows similarly as in the previous case, except in the last part: since $x$ is maximal in $P$, the label on $x$ will be incremented all the way to $q$.

\smallskip
\textbf{Case $x\notin S(f)$:} 
In this case, $jdt(f)(x)-1=f(x)-1$, since the only thing that happens to $x$ during the algorithm is that its label is decremented at the end. We wish to show that $\rho_{q-1}\circ\rho_{q-2}\circ\cdots\circ\rho_2\circ\rho_1(f)(x)=f(x)-1$. If $x$ covers no element with label $f(x)-1$ after the application of $\rho_{f(x)-2}\circ\cdots\rho_2\circ\rho_1$, then we are done, because $\rho_{f(x)-1}$ will decrement the label on $x$ and all other Bender Knuth involutions act trivially at $x$. 

Suppose $x$ covers an element or multiple elements with label $f(x)-1$ after the application of $\rho_{f(x)-2}\circ\cdots\rho_2\circ\rho_1$. Let $X_1$ be the set of all such elements. (Also, let $X_0=\{x\}$.)
Likewise, let $X_j$ be the set of all elements that are covered by some element of $X_{j-1}$ and have label $f(x)-j$ after the application of $\rho_{f(x)-j-1}\circ\cdots\rho_2\circ\rho_1$. 
The only way for the label on $x$ to not be decremented by $\rho_{f(x)-1}$ is if each $X_j$ is nonempty. 
For if some $X_j$ is empty, then no element of $X_{j-1}$ covers an element that has label $f(x)-j$ after the application of $\rho_{f(x)-j-1}\circ\cdots\rho_2\circ\rho_1$. Then, there is no possibility that any element of $X_{j-1}$ may be incremented by $\rho_{f(x)-j-2}$, thus $X_{j-1}$ must also be empty. Therefore, this implies $X_{k}$ is empty for all $k\leq j$. 

Assume the $X_j$ are all nonempty. Then consider the subposet $Y = \cup_{j} X_j$. Note that every element of $Y$ is comparable with (less than) $x$ in $P$. The minimal elements of $Y$ cannot initially have label $1$, since otherwise, all elements of $Y$ greater than the element labeled $1$ will be in the sliding subposet $S(f)$, and we had assumed $x\not\in S(f)$. Suppose no minimal element of $Y$ is labeled $1$. A minimal element of $Y$ labeled by $i>1$ will be decremented by $\rho_{i-1}$, thus this element cannot be in $Y$. 
So in either case, $\rho_{q-1}\circ\rho_{q-2}\circ\cdots\circ\rho_2\circ\rho_1(f)(x)=f(x)-1$.
\end{proof}

See Figures~\ref{fig:slide} through \ref{fig:ProResult} for an example.

\begin{figure}[htbp]
\begin{tikzpicture}
\begin{scope}[scale=1,xshift=0cm,yshift=0cm]
\coordinate (a1) at (1,0);
\coordinate (a2) at (2,0);
\coordinate (a3) at (3,0);
\coordinate (b1) at (.5,1);
\coordinate (b2) at (1.5,1);
\coordinate (b3) at (2.5,1);
\coordinate (b4) at (3.5,1);
\coordinate (c1) at (1,2);
\coordinate (c2) at (2,2);
\coordinate (c3) at (3,2);
\draw[] (a1) -- (b1);
\draw[] (a1) -- (b3);
\draw[] (a2) -- (b1);
\draw[] (a2) -- (b2);
\draw[] (a2) -- (b3);
\draw[] (a2) -- (b4);
\draw[] (a3) -- (b2);
\draw[] (a3) -- (b4);
\draw[] (b1) -- (c1);
\draw[] (b1) -- (c2);
\draw[] (b2) -- (c1);
\draw[] (b2) -- (c2);
\draw[] (b3) -- (c2);
\draw[] (b3) -- (c3);
\draw[] (b4) -- (c2);
\draw[] (b4) -- (c3);
\node[fill=white,draw,circle,inner sep=.5ex] at (a1) {1};
\node[fill=white,draw,circle,inner sep=.5ex] at (a2) {3};
\node[fill=white,draw,circle,inner sep=.5ex] at (a3) {1};
\node[fill=white,draw,circle,inner sep=.5ex] at (b1) {5};
\node[fill=white,draw,circle,inner sep=.5ex] at (b2) {4};
\node[fill=white,draw,circle,inner sep=.5ex] at (b3) {6};
\node[fill=white,draw,circle,inner sep=.5ex] at (b4) {4};
\node[fill=white,draw,circle,inner sep=.5ex] at (c1) {7};
\node[fill=white,draw,circle,inner sep=.5ex] at (c2) {8};
\node[fill=white,draw,circle,inner sep=.5ex] at (c3) {7};
\node[] at (4.2,1.6) {$\underrightarrow{\sigma_{1\rightarrow \boxempty}}$};
\end{scope}

\begin{scope}[scale=1,xshift=4.2cm,yshift=0cm]
\coordinate (a1) at (1,0);
\coordinate (a2) at (2,0);
\coordinate (a3) at (3,0);
\coordinate (b1) at (.5,1);
\coordinate (b2) at (1.5,1);
\coordinate (b3) at (2.5,1);
\coordinate (b4) at (3.5,1);
\coordinate (c1) at (1,2);
\coordinate (c2) at (2,2);
\coordinate (c3) at (3,2);
\draw[] (a1) -- (b1);
\draw[] (a1) -- (b3);
\draw[] (a2) -- (b1);
\draw[] (a2) -- (b2);
\draw[] (a2) -- (b3);
\draw[] (a2) -- (b4);
\draw[] (a3) -- (b2);
\draw[] (a3) -- (b4);
\draw[] (b1) -- (c1);
\draw[] (b1) -- (c2);
\draw[] (b2) -- (c1);
\draw[] (b2) -- (c2);
\draw[] (b3) -- (c2);
\draw[] (b3) -- (c3);
\draw[] (b4) -- (c2);
\draw[] (b4) -- (c3);
\node[fill=white,draw,circle,inner sep=.5ex] at (a1) {\phantom{a}};
\node[fill=white,draw,circle,inner sep=.5ex] at (a2) {3};
\node[fill=white,draw,circle,inner sep=.5ex] at (a3) {\phantom{a}};
\node[fill=white,draw,circle,inner sep=.5ex] at (b1) {5};
\node[fill=white,draw,circle,inner sep=.5ex] at (b2) {4};
\node[fill=white,draw,circle,inner sep=.5ex] at (b3) {6};
\node[fill=white,draw,circle,inner sep=.5ex] at (b4) {4};
\node[fill=white,draw,circle,inner sep=.5ex] at (c1) {7};
\node[fill=white,draw,circle,inner sep=.5ex] at (c2) {8};
\node[fill=white,draw,circle,inner sep=.5ex] at (c3) {7};
\node[] at (4.2,1.6) {$\underrightarrow{\sigma_{4}}$};
\end{scope}

\begin{scope}[scale=1,xshift=8.4cm,yshift=0cm]
\coordinate (a1) at (1,0);
\coordinate (a2) at (2,0);
\coordinate (a3) at (3,0);
\coordinate (b1) at (.5,1);
\coordinate (b2) at (1.5,1);
\coordinate (b3) at (2.5,1);
\coordinate (b4) at (3.5,1);
\coordinate (c1) at (1,2);
\coordinate (c2) at (2,2);
\coordinate (c3) at (3,2);
\draw[] (a1) -- (b1);
\draw[] (a1) -- (b3);
\draw[] (a2) -- (b1);
\draw[] (a2) -- (b2);
\draw[] (a2) -- (b3);
\draw[] (a2) -- (b4);
\draw[] (a3) -- (b2);
\draw[] (a3) -- (b4);
\draw[] (b1) -- (c1);
\draw[] (b1) -- (c2);
\draw[] (b2) -- (c1);
\draw[] (b2) -- (c2);
\draw[] (b3) -- (c2);
\draw[] (b3) -- (c3);
\draw[] (b4) -- (c2);
\draw[] (b4) -- (c3);
\node[fill=white,draw,circle,inner sep=.5ex] at (a1) {\phantom{a}};
\node[fill=white,draw,circle,inner sep=.5ex] at (a2) {3};
\node[fill=white,draw,circle,inner sep=.5ex] at (a3) {4};
\node[fill=white,draw,circle,inner sep=.5ex] at (b1) {5};
\node[fill=white,draw,circle,inner sep=.5ex] at (b2) {\phantom{a}};
\node[fill=white,draw,circle,inner sep=.5ex] at (b3) {6};
\node[fill=white,draw,circle,inner sep=.5ex] at (b4) {\phantom{a}};
\node[fill=white,draw,circle,inner sep=.5ex] at (c1) {7};
\node[fill=white,draw,circle,inner sep=.5ex] at (c2) {8};
\node[fill=white,draw,circle,inner sep=.5ex] at (c3) {7};
\node[] at (4.2,1.6) {$\underrightarrow{\sigma_{5}}$};
\end{scope}

\begin{scope}[scale=1,xshift=12.6cm,yshift=0cm]
\coordinate (a1) at (1,0);
\coordinate (a2) at (2,0);
\coordinate (a3) at (3,0);
\coordinate (b1) at (.5,1);
\coordinate (b2) at (1.5,1);
\coordinate (b3) at (2.5,1);
\coordinate (b4) at (3.5,1);
\coordinate (c1) at (1,2);
\coordinate (c2) at (2,2);
\coordinate (c3) at (3,2);
\draw[] (a1) -- (b1);
\draw[] (a1) -- (b3);
\draw[] (a2) -- (b1);
\draw[] (a2) -- (b2);
\draw[] (a2) -- (b3);
\draw[] (a2) -- (b4);
\draw[] (a3) -- (b2);
\draw[] (a3) -- (b4);
\draw[] (b1) -- (c1);
\draw[] (b1) -- (c2);
\draw[] (b2) -- (c1);
\draw[] (b2) -- (c2);
\draw[] (b3) -- (c2);
\draw[] (b3) -- (c3);
\draw[] (b4) -- (c2);
\draw[] (b4) -- (c3);
\node[fill=white,draw,circle,inner sep=.5ex] at (a1) {5};
\node[fill=white,draw,circle,inner sep=.5ex] at (a2) {3};
\node[fill=white,draw,circle,inner sep=.5ex] at (a3) {4};
\node[fill=white,draw,circle,inner sep=.5ex] at (b1) {\phantom{a}};
\node[fill=white,draw,circle,inner sep=.5ex] at (b2) {\phantom{a}};
\node[fill=white,draw,circle,inner sep=.5ex] at (b3) {6};
\node[fill=white,draw,circle,inner sep=.5ex] at (b4) {\phantom{a}};
\node[fill=white,draw,circle,inner sep=.5ex] at (c1) {7};
\node[fill=white,draw,circle,inner sep=.5ex] at (c2) {8};
\node[fill=white,draw,circle,inner sep=.5ex] at (c3) {7};
\end{scope}

\begin{scope}[scale=1,xshift=12.6cm,yshift=-4cm]
\coordinate (a1) at (1,0);
\coordinate (a2) at (2,0);
\coordinate (a3) at (3,0);
\coordinate (b1) at (.5,1);
\coordinate (b2) at (1.5,1);
\coordinate (b3) at (2.5,1);
\coordinate (b4) at (3.5,1);
\coordinate (c1) at (1,2);
\coordinate (c2) at (2,2);
\coordinate (c3) at (3,2);
\draw[] (a1) -- (b1);
\draw[] (a1) -- (b3);
\draw[] (a2) -- (b1);
\draw[] (a2) -- (b2);
\draw[] (a2) -- (b3);
\draw[] (a2) -- (b4);
\draw[] (a3) -- (b2);
\draw[] (a3) -- (b4);
\draw[] (b1) -- (c1);
\draw[] (b1) -- (c2);
\draw[] (b2) -- (c1);
\draw[] (b2) -- (c2);
\draw[] (b3) -- (c2);
\draw[] (b3) -- (c3);
\draw[] (b4) -- (c2);
\draw[] (b4) -- (c3);
\node[fill=white,draw,circle,inner sep=.5ex] at (a1) {5};
\node[fill=white,draw,circle,inner sep=.5ex] at (a2) {3};
\node[fill=white,draw,circle,inner sep=.5ex] at (a3) {4};
\node[fill=white,draw,circle,inner sep=.5ex] at (b1) {7};
\node[fill=white,draw,circle,inner sep=.5ex] at (b2) {7};
\node[fill=white,draw,circle,inner sep=.5ex] at (b3) {6};
\node[fill=white,draw,circle,inner sep=.5ex] at (b4) {7};
\node[fill=white,draw,circle,inner sep=.5ex] at (c1) {\phantom{a}};
\node[fill=white,draw,circle,inner sep=.5ex] at (c2) {8};
\node[fill=white,draw,circle,inner sep=.5ex] at (c3) {\phantom{a}};
\draw[thick,->] (2,3.3) -- (2,2.7) node[midway,left] {$\sigma_{7}$};
\end{scope}

\begin{scope}[scale=1,xshift=6.3cm,yshift=-4cm]
\coordinate (a1) at (1,0);
\coordinate (a2) at (2,0);
\coordinate (a3) at (3,0);
\coordinate (b1) at (.5,1);
\coordinate (b2) at (1.5,1);
\coordinate (b3) at (2.5,1);
\coordinate (b4) at (3.5,1);
\coordinate (c1) at (1,2);
\coordinate (c2) at (2,2);
\coordinate (c3) at (3,2);
\draw[] (a1) -- (b1);
\draw[] (a1) -- (b3);
\draw[] (a2) -- (b1);
\draw[] (a2) -- (b2);
\draw[] (a2) -- (b3);
\draw[] (a2) -- (b4);
\draw[] (a3) -- (b2);
\draw[] (a3) -- (b4);
\draw[] (b1) -- (c1);
\draw[] (b1) -- (c2);
\draw[] (b2) -- (c1);
\draw[] (b2) -- (c2);
\draw[] (b3) -- (c2);
\draw[] (b3) -- (c3);
\draw[] (b4) -- (c2);
\draw[] (b4) -- (c3);
\node[fill=white,draw,circle,inner sep=.5ex] at (a1) {5};
\node[fill=white,draw,circle,inner sep=.5ex] at (a2) {3};
\node[fill=white,draw,circle,inner sep=.5ex] at (a3) {4};
\node[fill=white,draw,circle,inner sep=.5ex] at (b1) {7};
\node[fill=white,draw,circle,inner sep=.5ex] at (b2) {7};
\node[fill=white,draw,circle,inner sep=.5ex] at (b3) {6};
\node[fill=white,draw,circle,inner sep=.5ex] at (b4) {7};
\node[fill=white,draw,circle,inner sep=.5ex] at (c1) {9};
\node[fill=white,draw,circle,inner sep=.5ex] at (c2) {8};
\node[fill=white,draw,circle,inner sep=.5ex] at (c3) {9};
\node[] at (5.1,1.6) {$\underleftarrow{\sigma_{\boxempty\rightarrow (q+1)}}$};
\end{scope}

\begin{scope}[scale=1,xshift=0cm,yshift=-4cm]
\coordinate (a1) at (1,0);
\coordinate (a2) at (2,0);
\coordinate (a3) at (3,0);
\coordinate (b1) at (.5,1);
\coordinate (b2) at (1.5,1);
\coordinate (b3) at (2.5,1);
\coordinate (b4) at (3.5,1);
\coordinate (c1) at (1,2);
\coordinate (c2) at (2,2);
\coordinate (c3) at (3,2);
\draw[] (a1) -- (b1);
\draw[] (a1) -- (b3);
\draw[] (a2) -- (b1);
\draw[] (a2) -- (b2);
\draw[] (a2) -- (b3);
\draw[] (a2) -- (b4);
\draw[] (a3) -- (b2);
\draw[] (a3) -- (b4);
\draw[] (b1) -- (c1);
\draw[] (b1) -- (c2);
\draw[] (b2) -- (c1);
\draw[] (b2) -- (c2);
\draw[] (b3) -- (c2);
\draw[] (b3) -- (c3);
\draw[] (b4) -- (c2);
\draw[] (b4) -- (c3);
\node[fill=white,draw,circle,inner sep=.5ex] at (a1) {4};
\node[fill=white,draw,circle,inner sep=.5ex] at (a2) {2};
\node[fill=white,draw,circle,inner sep=.5ex] at (a3) {3};
\node[fill=white,draw,circle,inner sep=.5ex] at (b1) {6};
\node[fill=white,draw,circle,inner sep=.5ex] at (b2) {6};
\node[fill=white,draw,circle,inner sep=.5ex] at (b3) {5};
\node[fill=white,draw,circle,inner sep=.5ex] at (b4) {6};
\node[fill=white,draw,circle,inner sep=.5ex] at (c1) {8};
\node[fill=white,draw,circle,inner sep=.5ex] at (c2) {7};
\node[fill=white,draw,circle,inner sep=.5ex] at (c3) {8};
\node[] at (5.1,1.6) {$\underleftarrow{\textrm{subtract 1}}$};
\end{scope}
\end{tikzpicture}
\caption{Example of $\jdtpro$ acting on an increasing labeling of a poset, showing only $\sigma_i$ acting non-trivially.}
\label{fig:slide}
\end{figure}
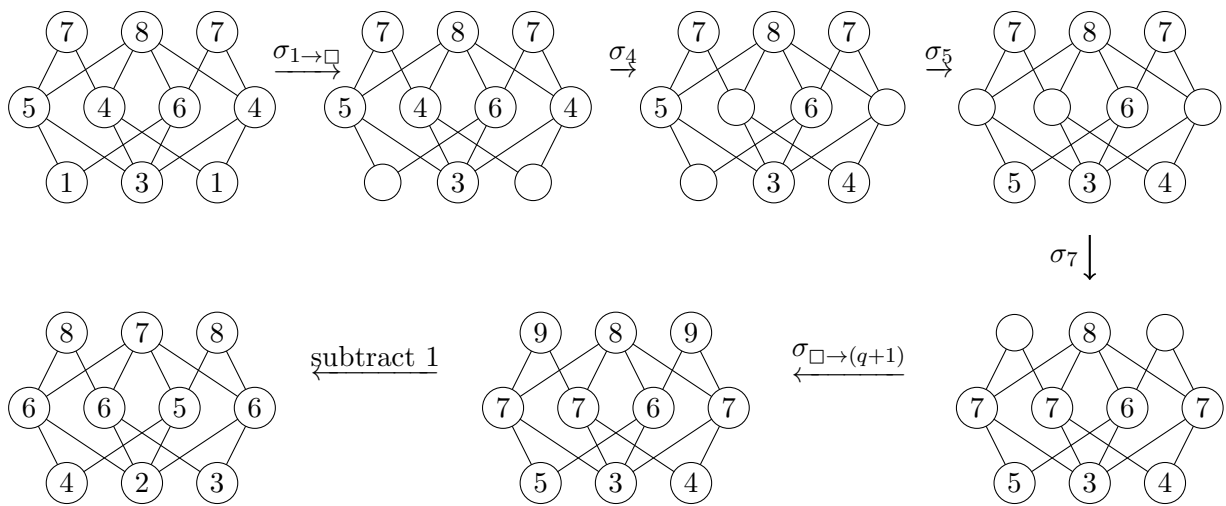

\begin{figure}[htbp]
\begin{tikzpicture}[scale=2]
\usetikzlibrary{decorations.markings}
\usetikzlibrary{arrows}
\coordinate (a1) at (1,0);
\coordinate (a2) at (2,0);
\coordinate (a3) at (3,0);
\coordinate (b1) at (.5,1);
\coordinate (b2) at (1.5,1);
\coordinate (b3) at (2.5,1);
\coordinate (b4) at (3.5,1);
\coordinate (c1) at (1,2);
\coordinate (c2) at (2,2);
\coordinate (c3) at (3,2);
\draw[very thick,decoration={markings,
    mark=at position 0.6 with {\arrow[scale=1.5]{latex}}},
    postaction=decorate] (a1) -- (b1);
\draw[] (a1) -- (b3);
\draw[] (a2) -- (b1);
\draw[] (a2) -- (b2);
\draw[] (a2) -- (b3);
\draw[] (a2) -- (b4);
\draw[very thick,decoration={markings,
    mark=at position 0.55 with {\arrow[scale=1.5]{latex}}},
    postaction=decorate] (a3) -- (b2);
\draw[very thick,decoration={markings,
    mark=at position 0.6 with {\arrow[scale=1.5]{latex}}},
    postaction=decorate] (a3) -- (b4);
\draw[very thick,decoration={markings,
    mark=at position 0.6 with {\arrow[scale=1.5]{latex}}},
    postaction=decorate] (b1) -- (c1);
\draw[] (b1) -- (c2);
\draw[very thick,decoration={markings,
    mark=at position 0.6 with {\arrow[scale=1.5]{latex}}},
    postaction=decorate] (b2) -- (c1);
\draw[] (b2) -- (c2);
\draw[] (b3) -- (c2);
\draw[] (b3) -- (c3);
\draw[] (b4) -- (c2);
\draw[very thick,decoration={markings,
    mark=at position 0.6 with {\arrow[scale=1.5]{latex}}},
    postaction=decorate] (b4) -- (c3);
\node[fill=white,draw,circle,inner sep=.5ex] at (a1) {1};
\node[fill=white,draw,circle,inner sep=.5ex] at (a2) {3};
\node[fill=white,draw,circle,inner sep=.5ex] at (a3) {1};
\node[fill=white,draw,circle,inner sep=.5ex] at (b1) {5};
\node[fill=white,draw,circle,inner sep=.5ex] at (b2) {4};
\node[fill=white,draw,circle,inner sep=.5ex] at (b3) {6};
\node[fill=white,draw,circle,inner sep=.5ex] at (b4) {4};
\node[fill=white,draw,circle,inner sep=.5ex] at (c1) {7};
\node[fill=white,draw,circle,inner sep=.5ex] at (c2) {8};
\node[fill=white,draw,circle,inner sep=.5ex] at (c3) {7};
\end{tikzpicture}
\caption{An increasing labeling of a poset with $q=8$; the sliding subposet when performing jeu de taquin promotion is indicated by edges with arrows.}
\label{fig:slide2}
\end{figure}
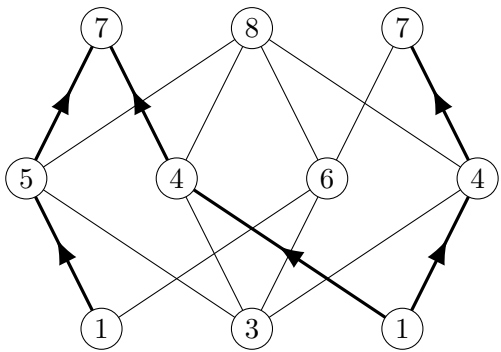

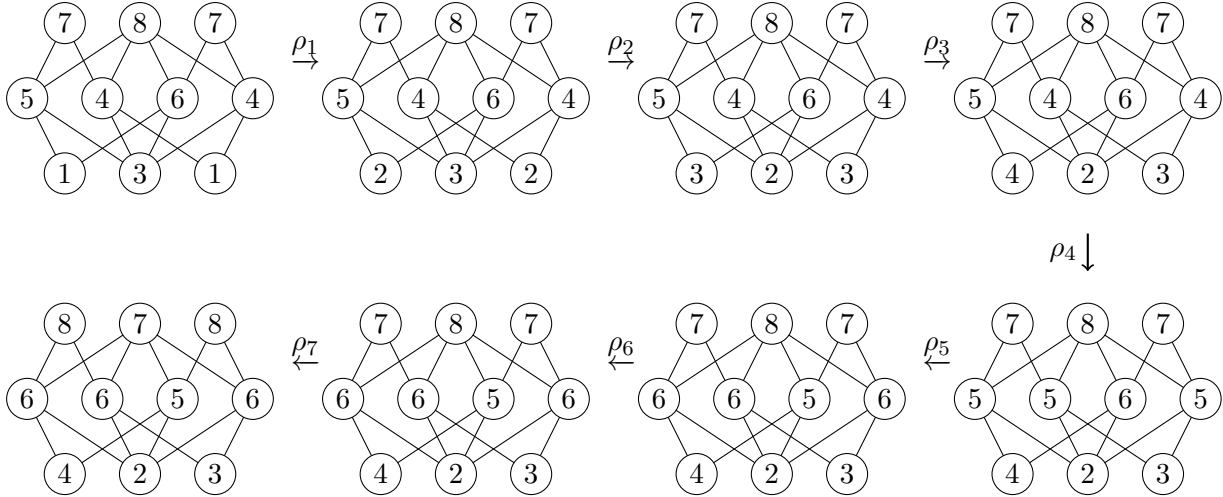
\begin{figure}[htbp]
\begin{tikzpicture}
\begin{scope}[scale=1,xshift=0cm,yshift=0cm]
\coordinate (a1) at (1,0);
\coordinate (a2) at (2,0);
\coordinate (a3) at (3,0);
\coordinate (b1) at (.5,1);
\coordinate (b2) at (1.5,1);
\coordinate (b3) at (2.5,1);
\coordinate (b4) at (3.5,1);
\coordinate (c1) at (1,2);
\coordinate (c2) at (2,2);
\coordinate (c3) at (3,2);
\draw[] (a1) -- (b1);
\draw[] (a1) -- (b3);
\draw[] (a2) -- (b1);
\draw[] (a2) -- (b2);
\draw[] (a2) -- (b3);
\draw[] (a2) -- (b4);
\draw[] (a3) -- (b2);
\draw[] (a3) -- (b4);
\draw[] (b1) -- (c1);
\draw[] (b1) -- (c2);
\draw[] (b2) -- (c1);
\draw[] (b2) -- (c2);
\draw[] (b3) -- (c2);
\draw[] (b3) -- (c3);
\draw[] (b4) -- (c2);
\draw[] (b4) -- (c3);
\node[fill=white,draw,circle,inner sep=.5ex] at (a1) {1};
\node[fill=white,draw,circle,inner sep=.5ex] at (a2) {3};
\node[fill=white,draw,circle,inner sep=.5ex] at (a3) {1};
\node[fill=white,draw,circle,inner sep=.5ex] at (b1) {5};
\node[fill=white,draw,circle,inner sep=.5ex] at (b2) {4};
\node[fill=white,draw,circle,inner sep=.5ex] at (b3) {6};
\node[fill=white,draw,circle,inner sep=.5ex] at (b4) {4};
\node[fill=white,draw,circle,inner sep=.5ex] at (c1) {7};
\node[fill=white,draw,circle,inner sep=.5ex] at (c2) {8};
\node[fill=white,draw,circle,inner sep=.5ex] at (c3) {7};
\node[] at (4.2,1.6) {$\underrightarrow{\rho_{1}}$};
\end{scope}

\begin{scope}[scale=1,xshift=4.2cm,yshift=0cm]
\coordinate (a1) at (1,0);
\coordinate (a2) at (2,0);
\coordinate (a3) at (3,0);
\coordinate (b1) at (.5,1);
\coordinate (b2) at (1.5,1);
\coordinate (b3) at (2.5,1);
\coordinate (b4) at (3.5,1);
\coordinate (c1) at (1,2);
\coordinate (c2) at (2,2);
\coordinate (c3) at (3,2);
\draw[] (a1) -- (b1);
\draw[] (a1) -- (b3);
\draw[] (a2) -- (b1);
\draw[] (a2) -- (b2);
\draw[] (a2) -- (b3);
\draw[] (a2) -- (b4);
\draw[] (a3) -- (b2);
\draw[] (a3) -- (b4);
\draw[] (b1) -- (c1);
\draw[] (b1) -- (c2);
\draw[] (b2) -- (c1);
\draw[] (b2) -- (c2);
\draw[] (b3) -- (c2);
\draw[] (b3) -- (c3);
\draw[] (b4) -- (c2);
\draw[] (b4) -- (c3);
\node[fill=white,draw,circle,inner sep=.5ex] at (a1) {2};
\node[fill=white,draw,circle,inner sep=.5ex] at (a2) {3};
\node[fill=white,draw,circle,inner sep=.5ex] at (a3) {2};
\node[fill=white,draw,circle,inner sep=.5ex] at (b1) {5};
\node[fill=white,draw,circle,inner sep=.5ex] at (b2) {4};
\node[fill=white,draw,circle,inner sep=.5ex] at (b3) {6};
\node[fill=white,draw,circle,inner sep=.5ex] at (b4) {4};
\node[fill=white,draw,circle,inner sep=.5ex] at (c1) {7};
\node[fill=white,draw,circle,inner sep=.5ex] at (c2) {8};
\node[fill=white,draw,circle,inner sep=.5ex] at (c3) {7};
\node[] at (4.2,1.6) {$\underrightarrow{\rho_{2}}$};
\end{scope}

\begin{scope}[scale=1,xshift=8.4cm,yshift=0cm]
\coordinate (a1) at (1,0);
\coordinate (a2) at (2,0);
\coordinate (a3) at (3,0);
\coordinate (b1) at (.5,1);
\coordinate (b2) at (1.5,1);
\coordinate (b3) at (2.5,1);
\coordinate (b4) at (3.5,1);
\coordinate (c1) at (1,2);
\coordinate (c2) at (2,2);
\coordinate (c3) at (3,2);
\draw[] (a1) -- (b1);
\draw[] (a1) -- (b3);
\draw[] (a2) -- (b1);
\draw[] (a2) -- (b2);
\draw[] (a2) -- (b3);
\draw[] (a2) -- (b4);
\draw[] (a3) -- (b2);
\draw[] (a3) -- (b4);
\draw[] (b1) -- (c1);
\draw[] (b1) -- (c2);
\draw[] (b2) -- (c1);
\draw[] (b2) -- (c2);
\draw[] (b3) -- (c2);
\draw[] (b3) -- (c3);
\draw[] (b4) -- (c2);
\draw[] (b4) -- (c3);
\node[fill=white,draw,circle,inner sep=.5ex] at (a1) {3};
\node[fill=white,draw,circle,inner sep=.5ex] at (a2) {2};
\node[fill=white,draw,circle,inner sep=.5ex] at (a3) {3};
\node[fill=white,draw,circle,inner sep=.5ex] at (b1) {5};
\node[fill=white,draw,circle,inner sep=.5ex] at (b2) {4};
\node[fill=white,draw,circle,inner sep=.5ex] at (b3) {6};
\node[fill=white,draw,circle,inner sep=.5ex] at (b4) {4};
\node[fill=white,draw,circle,inner sep=.5ex] at (c1) {7};
\node[fill=white,draw,circle,inner sep=.5ex] at (c2) {8};
\node[fill=white,draw,circle,inner sep=.5ex] at (c3) {7};
\node[] at (4.2,1.6) {$\underrightarrow{\rho_{3}}$};
\end{scope}

\begin{scope}[scale=1,xshift=12.6cm,yshift=0cm]
\coordinate (a1) at (1,0);
\coordinate (a2) at (2,0);
\coordinate (a3) at (3,0);
\coordinate (b1) at (.5,1);
\coordinate (b2) at (1.5,1);
\coordinate (b3) at (2.5,1);
\coordinate (b4) at (3.5,1);
\coordinate (c1) at (1,2);
\coordinate (c2) at (2,2);
\coordinate (c3) at (3,2);
\draw[] (a1) -- (b1);
\draw[] (a1) -- (b3);
\draw[] (a2) -- (b1);
\draw[] (a2) -- (b2);
\draw[] (a2) -- (b3);
\draw[] (a2) -- (b4);
\draw[] (a3) -- (b2);
\draw[] (a3) -- (b4);
\draw[] (b1) -- (c1);
\draw[] (b1) -- (c2);
\draw[] (b2) -- (c1);
\draw[] (b2) -- (c2);
\draw[] (b3) -- (c2);
\draw[] (b3) -- (c3);
\draw[] (b4) -- (c2);
\draw[] (b4) -- (c3);
\node[fill=white,draw,circle,inner sep=.5ex] at (a1) {4};
\node[fill=white,draw,circle,inner sep=.5ex] at (a2) {2};
\node[fill=white,draw,circle,inner sep=.5ex] at (a3) {3};
\node[fill=white,draw,circle,inner sep=.5ex] at (b1) {5};
\node[fill=white,draw,circle,inner sep=.5ex] at (b2) {4};
\node[fill=white,draw,circle,inner sep=.5ex] at (b3) {6};
\node[fill=white,draw,circle,inner sep=.5ex] at (b4) {4};
\node[fill=white,draw,circle,inner sep=.5ex] at (c1) {7};
\node[fill=white,draw,circle,inner sep=.5ex] at (c2) {8};
\node[fill=white,draw,circle,inner sep=.5ex] at (c3) {7};
\end{scope}

\begin{scope}[scale=1,xshift=12.6cm,yshift=-4cm]
\coordinate (a1) at (1,0);
\coordinate (a2) at (2,0);
\coordinate (a3) at (3,0);
\coordinate (b1) at (.5,1);
\coordinate (b2) at (1.5,1);
\coordinate (b3) at (2.5,1);
\coordinate (b4) at (3.5,1);
\coordinate (c1) at (1,2);
\coordinate (c2) at (2,2);
\coordinate (c3) at (3,2);
\draw[] (a1) -- (b1);
\draw[] (a1) -- (b3);
\draw[] (a2) -- (b1);
\draw[] (a2) -- (b2);
\draw[] (a2) -- (b3);
\draw[] (a2) -- (b4);
\draw[] (a3) -- (b2);
\draw[] (a3) -- (b4);
\draw[] (b1) -- (c1);
\draw[] (b1) -- (c2);
\draw[] (b2) -- (c1);
\draw[] (b2) -- (c2);
\draw[] (b3) -- (c2);
\draw[] (b3) -- (c3);
\draw[] (b4) -- (c2);
\draw[] (b4) -- (c3);
\node[fill=white,draw,circle,inner sep=.5ex] at (a1) {4};
\node[fill=white,draw,circle,inner sep=.5ex] at (a2) {2};
\node[fill=white,draw,circle,inner sep=.5ex] at (a3) {3};
\node[fill=white,draw,circle,inner sep=.5ex] at (b1) {5};
\node[fill=white,draw,circle,inner sep=.5ex] at (b2) {5};
\node[fill=white,draw,circle,inner sep=.5ex] at (b3) {6};
\node[fill=white,draw,circle,inner sep=.5ex] at (b4) {5};
\node[fill=white,draw,circle,inner sep=.5ex] at (c1) {7};
\node[fill=white,draw,circle,inner sep=.5ex] at (c2) {8};
\node[fill=white,draw,circle,inner sep=.5ex] at (c3) {7};
\draw[thick,->] (2,3.2) -- (2,2.7) node[midway,left] {$\rho_{4}$};
\end{scope}

\begin{scope}[scale=1,xshift=8.4cm,yshift=-4cm]
\coordinate (a1) at (1,0);
\coordinate (a2) at (2,0);
\coordinate (a3) at (3,0);
\coordinate (b1) at (.5,1);
\coordinate (b2) at (1.5,1);
\coordinate (b3) at (2.5,1);
\coordinate (b4) at (3.5,1);
\coordinate (c1) at (1,2);
\coordinate (c2) at (2,2);
\coordinate (c3) at (3,2);
\draw[] (a1) -- (b1);
\draw[] (a1) -- (b3);
\draw[] (a2) -- (b1);
\draw[] (a2) -- (b2);
\draw[] (a2) -- (b3);
\draw[] (a2) -- (b4);
\draw[] (a3) -- (b2);
\draw[] (a3) -- (b4);
\draw[] (b1) -- (c1);
\draw[] (b1) -- (c2);
\draw[] (b2) -- (c1);
\draw[] (b2) -- (c2);
\draw[] (b3) -- (c2);
\draw[] (b3) -- (c3);
\draw[] (b4) -- (c2);
\draw[] (b4) -- (c3);
\node[fill=white,draw,circle,inner sep=.5ex] at (a1) {4};
\node[fill=white,draw,circle,inner sep=.5ex] at (a2) {2};
\node[fill=white,draw,circle,inner sep=.5ex] at (a3) {3};
\node[fill=white,draw,circle,inner sep=.5ex] at (b1) {6};
\node[fill=white,draw,circle,inner sep=.5ex] at (b2) {6};
\node[fill=white,draw,circle,inner sep=.5ex] at (b3) {5};
\node[fill=white,draw,circle,inner sep=.5ex] at (b4) {6};
\node[fill=white,draw,circle,inner sep=.5ex] at (c1) {7};
\node[fill=white,draw,circle,inner sep=.5ex] at (c2) {8};
\node[fill=white,draw,circle,inner sep=.5ex] at (c3) {7};
\node[] at (4.2,1.6) {$\underleftarrow{\rho_5}$};
\end{scope}

\begin{scope}[scale=1,xshift=4.2cm,yshift=-4cm]
\coordinate (a1) at (1,0);
\coordinate (a2) at (2,0);
\coordinate (a3) at (3,0);
\coordinate (b1) at (.5,1);
\coordinate (b2) at (1.5,1);
\coordinate (b3) at (2.5,1);
\coordinate (b4) at (3.5,1);
\coordinate (c1) at (1,2);
\coordinate (c2) at (2,2);
\coordinate (c3) at (3,2);
\draw[] (a1) -- (b1);
\draw[] (a1) -- (b3);
\draw[] (a2) -- (b1);
\draw[] (a2) -- (b2);
\draw[] (a2) -- (b3);
\draw[] (a2) -- (b4);
\draw[] (a3) -- (b2);
\draw[] (a3) -- (b4);
\draw[] (b1) -- (c1);
\draw[] (b1) -- (c2);
\draw[] (b2) -- (c1);
\draw[] (b2) -- (c2);
\draw[] (b3) -- (c2);
\draw[] (b3) -- (c3);
\draw[] (b4) -- (c2);
\draw[] (b4) -- (c3);
\node[fill=white,draw,circle,inner sep=.5ex] at (a1) {4};
\node[fill=white,draw,circle,inner sep=.5ex] at (a2) {2};
\node[fill=white,draw,circle,inner sep=.5ex] at (a3) {3};
\node[fill=white,draw,circle,inner sep=.5ex] at (b1) {6};
\node[fill=white,draw,circle,inner sep=.5ex] at (b2) {6};
\node[fill=white,draw,circle,inner sep=.5ex] at (b3) {5};
\node[fill=white,draw,circle,inner sep=.5ex] at (b4) {6};
\node[fill=white,draw,circle,inner sep=.5ex] at (c1) {7};
\node[fill=white,draw,circle,inner sep=.5ex] at (c2) {8};
\node[fill=white,draw,circle,inner sep=.5ex] at (c3) {7};
\node[] at (4.2,1.6) {$\underleftarrow{\rho_6}$};
\end{scope}

\begin{scope}[scale=1,xshift=0cm,yshift=-4cm]
\coordinate (a1) at (1,0);
\coordinate (a2) at (2,0);
\coordinate (a3) at (3,0);
\coordinate (b1) at (.5,1);
\coordinate (b2) at (1.5,1);
\coordinate (b3) at (2.5,1);
\coordinate (b4) at (3.5,1);
\coordinate (c1) at (1,2);
\coordinate (c2) at (2,2);
\coordinate (c3) at (3,2);
\draw[] (a1) -- (b1);
\draw[] (a1) -- (b3);
\draw[] (a2) -- (b1);
\draw[] (a2) -- (b2);
\draw[] (a2) -- (b3);
\draw[] (a2) -- (b4);
\draw[] (a3) -- (b2);
\draw[] (a3) -- (b4);
\draw[] (b1) -- (c1);
\draw[] (b1) -- (c2);
\draw[] (b2) -- (c1);
\draw[] (b2) -- (c2);
\draw[] (b3) -- (c2);
\draw[] (b3) -- (c3);
\draw[] (b4) -- (c2);
\draw[] (b4) -- (c3);
\node[fill=white,draw,circle,inner sep=.5ex] at (a1) {4};
\node[fill=white,draw,circle,inner sep=.5ex] at (a2) {2};
\node[fill=white,draw,circle,inner sep=.5ex] at (a3) {3};
\node[fill=white,draw,circle,inner sep=.5ex] at (b1) {6};
\node[fill=white,draw,circle,inner sep=.5ex] at (b2) {6};
\node[fill=white,draw,circle,inner sep=.5ex] at (b3) {5};
\node[fill=white,draw,circle,inner sep=.5ex] at (b4) {6};
\node[fill=white,draw,circle,inner sep=.5ex] at (c1) {8};
\node[fill=white,draw,circle,inner sep=.5ex] at (c2) {7};
\node[fill=white,draw,circle,inner sep=.5ex] at (c3) {8};
\node[] at (4.2,1.6) {$\underleftarrow{\rho_7}$};
\end{scope}

\end{tikzpicture}
\caption{Step by step toggle promotion on the same increasing labeling as in Figure~\ref{fig:slide}. }
\label{fig:ProResult}
\end{figure}

See Figures~\ref{fig:slide},\ref{fig:slide2},\ref{fig:ProResult} for an example.

Definition~\ref{def:jdt} implies the following lemma, which we use to prove a new resonance statement in Theorem~\ref{thm:res}. Define the \emph{binary content} of an increasing labeling $f\in \inc{q}{P}$ to be the sequence $\textrm{Con}(f)=(a_1, a_2, \ldots, a_q)$, where $a_i=1$ if $f(p)=i$ for some $p\in P$ and $0$ otherwise.
\begin{lemma}
Let $f\in\inc{q}{P}$. If $\textrm{Con}(f) = (a_1, a_2, \ldots, a_q)$, then
$\textrm{Con}(\incpro(f))$ is the cyclic shift $(a_2, \ldots, a_q, a_1)$. That is, $\jdtpro$ on $\inc{q}{P}$ rotates the binary content vector. 
\end{lemma}
\begin{proof}
Using Definition~\ref{def:jdt}, we see that if $f\in \inc{q}{P}$ and $1\notin f(P)$, then $\jdtpro(f(x))=f(x)-1$ for all $x\in P$, so that $q\notin f(P)$. If $k\notin f(P)$, for $k\geq 2$, then when we subtract $1$ from all entries at the end of the sliding algorithm, we will have $k-1\notin \jdtpro(f(P))$.  
\end{proof}

This lemma and Theorem~\ref{thm:bk=jdt} yields the following resonance statement, which is an analogue of \cite[Theorem 2.2]{DPS2015} in the case of increasing tableaux. 
\begin{corollary}
\label{thm:res}
$(\inc{q}{P},\langle\incpro\rangle,\mbox{\rm Con})$ exhibits resonance with frequency $q$.
\end{corollary}

So though the order of $\incpro$ on $\inc{q}{P}$ may be large, the action projects to a cyclic shift of order $q$.



\section{Equivariance of the bijection}
\label{sec:eq_bij}
The purpose of this section will be to prove the following theorem, which yields the following corollary.
\begin{theorem}
\label{thm:prorow}
When $H_{\Gamma}$ is a column toggle order,
there is an equivariant bijection between $\inc{R}{P}$ under promotion and order ideals in $\Gamma(P,R)$ under $\row$.
\end{theorem}

\begin{corollary}
\label{cor:prorow}
There is an equivariant bijection between $\inc{q}{P}$ under promotion and order ideals in $\Gamma(P,q)$ under $\row$.
\end{corollary}

There are two ingredients to the proof of this theorem. In Subsection~\ref{row=pro}, we show that when a poset can be partitioned by $H:P\rightarrow\mathbb{Z}$ into \emph{columns} satisfying a certain property, this partition induces a toggle group action $\pro_H$ which we prove in Theorem~\ref{thm:newprorow} is conjugate to rowmotion. This generalizes the notion of hyperplane toggles from~\cite{DPS2015} and applies to a larger class of posets; we discuss some examples in Subsection~\ref{sec:app}. In Subsection~\ref{pro=pro}, we show that $\Gamma(P,q)$ naturally has such a column toggle order, and that toggle-promotion on $\Gamma(P,q)$ with respect to this column toggle order exactly corresponds to Bender-Knuth promotion on $\inc{q}{P}$. We then prove Theorem~\ref{thm:prorow}.

\subsection{Toggle-promotion is conjugate to rowmotion}
\label{row=pro}
In this section, we define a toggle group action that toggles every element of the poset exactly once using a \emph{toggle order}. A toggle order does not specify a total ordering on the poset in which elements must be toggled, but it allows elements that are not part of a covering relation to be toggled simultaneously.
In the specific case where this toggle order is a \emph{column toggle order}, we show in Theorem~\ref{thm:newprorow} this toggle group action is conjugate to rowmotion.

Note that as opposed to previous results establishing the conjugacy of rowmotion and various promotion toggle group actions in \cite{SW2012} and \cite{DPS2015}, we do not require $P$ to be ranked, and our constructions do not rely on any kind of geometric embedding.

\begin{definition}
\label{TogOrder}
We say that a function $H:P\rightarrow \mathbb{Z}$ is a \emph{toggle order} if $p_1\lessdot p_2$ implies $H(p_1)\neq H(p_2)$. 
Given a toggle order $H$, define $T_H^i$ to be the toggle group action that is the product of all $t_p$ for $p\in P$ such that $H(p)=i$.
\end{definition}

\begin{lemma}
For a toggle order $H$, $T_H^i$ is well-defined, and $(T_H^i)^2=1$.
\end{lemma}

\begin{proof}
Two toggles commute with each other as long as they are not part of a covering relation (see Remark~\ref{rmk:ToggleCommute}). Since we restrict a toggle order from mapping two elements in a covering relation to the same number, the toggles corresponding to all $p\in P$ with $H(p)=i$ pairwise commute, and thus $T_H^i$ is well-defined. Since each of the individual toggles are involutions and they pairwise commute, their product is also an involution.
\end{proof}

\begin{definition}
\label{def:TogPro}
We say that \emph{toggle-promotion} with respect to a toggle order $H$, denoted $\togpro_H$, is the toggle group action given by \[ \ldots T_H^{2}T_H^{1}T_H^0 T_H^{-1} T_H^{-2}\ldots \]
\end{definition}

Note that every element of $P$ is toggled exactly once in $\togpro_H$.
We now consider a special kind of toggle order.

\begin{definition}
\label{def:ColTog}
We say that a function $H:P\rightarrow \mathbb{Z}$ is a \emph{column toggle order} if whenever $p_1\lessdot p_2$ in $P$, then $H(p_1)=H(p_2)\pm 1$.
\end{definition}

We call this a column toggle order because it implies that our poset elements can be partitioned into subsets, called \emph{columns}, whose elements have covering relations only with elements in adjacent columns. We can also think of it as inducing a bipartite coloring of the Hasse diagram of $P$.

\begin{lemma}
\label{lem:TogCom}
Given a column toggle order, all elements of $p$ with $H(p)$ even (resp. odd) commute which each other.
\end{lemma}

\begin{proof}
This is a consequence of Remark~\ref{rmk:ToggleCommute}, indicating that only elements in a covering relation can fail to commute, and Definition~\ref{def:ColTog}, requiring that elements in a covering relation have different parity.
\end{proof}

\begin{remark}
Definition~\ref{def:ColTog} generalizes the columns of rc-posets from \cite{SW2012} and hyperplane toggles from \cite{DPS2015}, as these are both examples of column toggle orders. Therefore, the upcoming Theorem~\ref{thm:newprorow} is a generalization of the promotion and rowmotion theorems of \cite{SW2012} and \cite{DPS2015}, which we stated in Section~\ref{sec:background} as Theorems~\ref{thm:ProRowEq} and \ref{thm:DPS}. Note that Theorem~\ref{thm:newprorow} applies to non-ranked posets, while in the previous cases, the posets were required to be ranked.
\end{remark}

We will show that toggling with respect to a column toggle order is conjugate to rowmotion. Then in Subsection~\ref{sec:app}, we will show that $\Gamma(P,q)$ and Cartesian products of ranked posets have a natural column toggle order.

\begin{lemma}
Let $H$ be a column toggle order. Then $T_H^i$ and $T_H^j$ commute if $|i-j|>1$.
\end{lemma}

\begin{proof}
This follows directly from the definition of a column toggle order. Two toggles $t_a$ and $t_b$ can only fail to commute if they are part of a covering relation, and if $a\lessdot b$ then $t_a$ and $t_b$ are sent to adjacent columns $T_H^i$ and $T_H^{i+1}$ for some $i$.
\end{proof}

\begin{definition}
Let $H$ be a column toggle order on $P$.
Let $L_1$ be the set of all minimal elements $p$ of $P$ with $H(p)$ being odd, and for $i>1$ recursively define $L_i$ to be the minimal elements of $P\setminus\displaystyle\cup_{j<i} L_j$.
Then we define $R_H^i$ to be the product of all toggles $t_p$ with $p\in L_i$.
\end{definition}

One can see that the map taking $p\in L_i$ to $i$ defines a column toggle order. Thus, the $R_H^i$'s share the same properties as the $T_H^i$'s.

\begin{lemma}
Let $H$ be a column toggle order. Then $R_H^i$ is well-defined, $(R_H^i)^2=1$, and $R_H^i$ commutes with $R_H^j$ if $|i-j|>1$.
\end{lemma}

\begin{lemma}
$\mathrm{Row}_H:=R_H^1\ldots R_H^c$ is $\row$.
\end{lemma}
\begin{proof}
We can find a linear extension $f$ of a poset by labeling the elements of $L_1$ with 1 through $|L_1|$, the elements of $L_2$ with $|L_1| + 1$ through $|L_1| + |L_2|$, and so forth. $R_H^1\ldots R_H^c$ sweeps through the poset in the reverse order of $f$. By Theorem \ref{thm:linextrow}, this is rowmotion.
\end{proof}

\begin{lemma}
\label{lem:parity}
For a column order $H$, let $i$ and $j$ be such that $p$ is toggled in both $R_H^i$ and $T_H^j$. Then $i$ and $j$ have the same parity.
\end{lemma}

\begin{proof}
This follows from the construction of the $R_H^i$'s. 
We start by taking $L_1$ to be minimal elements with $H(p)$ odd. So if $p$ is toggled in $R_H^1$, $p$ is in $L_1$, so $H(p)$ is odd, and $p$ is toggled in a $T_H^j$ with $j$ odd.
Elements of $L_2$ will either be minimal elements of $P$ with $H(p)$ being even, or will be upper covers of minimal elements with $H(p)$ being odd. Since following a covering relation in $P$ changes the parity of $H(p)$, all elements of $L_2$ will have $H(p)$ even.
From then on, every element in $L_i$ will have some element of $L_{i-1}$ as a lower cover. So by induction, if all elements of $L_{i-1}$ have $H(p)$ with parity matching $i-1$, then all elements of $L_i$ will have parity matching $i$.
\end{proof}

\begin{definition}
We say that the \emph{promotion-support} of $H$ is the smallest interval $[a,b]$ in $\mathbb{Z}$ containing the image of $P$ under $H$. We say the \emph{rowmotion-support} is $[1,c]$, where $c$ is the largest index for which $L_c$ is non-empty.
\end{definition}

\begin{definition}
Let \emph{gyration}, $\mathrm{Gyr}_H$ be the toggle group element which first toggles all elements $p$ of $P$ with $H(p)$ even, then all elements with $H(p)$ being odd.
\end{definition}

Recall that by Lemma~\ref{lem:TogCom}, we do not need to specify the order of the elements with $H(p)$ even nor the order of the elements with $H(p)$ odd, as toggles for elements with $H(p)$ having the same parity pairwise commute.

\begin{lemma}
\label{lem:gyr}
If the promotion-support of $H$ is $[a,b]$, then for any bijection $\sigma:[a,b]\rightarrow[a,b]$ such that $\sigma(k)$ is odd if $k<\frac{a+b}{2}$ and even if $k>\frac{a+b}{2}$, we have $T_H^{\sigma(a)}\ldots T_H^{\sigma(b)}=\mathrm{Gyr}_H$.
\end{lemma}

As a result of Lemma~\ref{lem:parity}, we also obtain the following result.

\begin{lemma}
\label{lem:gyrR}
If the rowmotion-support of $H$ is $[1,c]$, then for any bijection $\sigma:[1,c]\rightarrow[1,c]$ such that $\sigma(k)$ is odd if $k<\frac{c+1}{2}$ and even if $k>\frac{c+1}{2}$, we have $R_H^{\sigma(a)}\ldots R_H^{\sigma(b)}=\mathrm{Gyr}_H$.
\end{lemma}

We use the following group theory lemma from \cite{Humphreys1990}, which appears as Lemma 5.1 in \cite{SW2012}.

\begin{lemma}[\protect{\cite[p.\ 74]{Humphreys1990}}]
\label{lem:conjugatelem}
Let $G$ be a group whose generators $g_1,\ldots, g_n$ satisfy $g_i^2 = 1$ and $(g_i g_j)^2 = 1$ if
$|i - j| > 1$. Then for any $\sigma, \tau \in \mathfrak{S}_n$,
$\prod_i g_{\sigma(i)}$ and $\prod_i g_{\tau(i)}$ are conjugate.
\end{lemma}

Now, we may prove the following theorem.

\begin{theorem}
\label{thm:newprorow}
Let $H$ be a column toggle order of $P$. Then the toggle group action $\togpro_{H}$ on $J(P)$ is conjugate to $\row$ on $J(P)$. 
\end{theorem}

\begin{proof}
First, we let $g_i$ be $T_H^i$ and apply Lemma~\ref{lem:conjugatelem} to the identity permutation and $\sigma_1:[a,b]\rightarrow[a,b]$ such that $\sigma_1(k)$ is odd if $k<\frac{a+b}{2}$ and even if $k>\frac{a+b}{2}$. Then by Lemma \ref{lem:gyr}, we see that $\togpro_H$ is conjugate to $\mathrm{Gyr}_H$.

Similarly, we can let $g_i$ be $R_H^i$ and apply Lemma~\ref{lem:conjugatelem} to the identity permutation and $\sigma_2:[1,c]\rightarrow[1,c]$ such that $\sigma_2(k)$ is odd if $k<\frac{c+1}{2}$ and even if $k>\frac{c+1}{2}$. Then by Lemma \ref{lem:gyrR}, we see that rowmotion is also is conjugate to $\mathrm{Gyr}_H$.

Thus, rowmotion and toggle-promotion are both conjugate to gyration, and thus conjugate to each other.
\end{proof}

\begin{corollary}
Given two column toggle orders $H_1$ and $H_2$ of $P$, $\togpro_{H_1}$ and $\togpro_{H_2}$ are both conjugate to $\row$, and thus conjugate to each other.
\end{corollary}

\subsection{Applications of the conjugacy of toggle-promotion and rowmotion}
\label{sec:app}
As our first application of Theorem~\ref{thm:newprorow}, we consider $\Gamma(P,R)$.

\begin{lemma}
\label{lem:TogPro}
For any $\Gamma(P,R)$, the map $H_{\Gamma}:\Gamma(P,R)\rightarrow\mathbb{Z}$ taking $(p,k)$ to $k$
defines a  toggle order.
\end{lemma}

\begin{proof}
This follows from the covering relations given in Definition~\ref{def:GammaOne}, as we can never have a covering relation $(p_1,k)\lessdot (p_2,k)$.
\end{proof}

Since the construction of $\Gamma$ gives a natural toggle order, we may define toggle-promotion with respect to this toggle order. 
Lemma~\ref{lem:TogPro} and Theorem~\ref{thm:newprorow} yield the following corollary.

\begin{corollary}
\label{cor:togcolprorow}
If $H_{\Gamma}$ is a column toggle order, then $\togpro_{H_{\Gamma}}$ on $J(\Gamma(P,R))$ is conjugate to $\row$.
\end{corollary}

In Theorem~\ref{thm:proequalspro}, we will show that $\togpro_{H_{\Gamma}}$ on $J(\Gamma(P,R))$ exactly corresponds to increasing labeling promotion $\incpro$ on $\inc{R}{P}$.

\smallskip
We obtain a stronger result when we look at the case where the range of values for each entry is an interval. Note that $\Gamma(P,q)$ is one such example.

\begin{lemma}
\label{lem:GammaColTog}
If a consistent restriction function $R$ always has $R(p)$ a non-empty interval, then for $\Gamma(P,R)$, the map $H_{\Gamma}$ defines a column toggle order.
\end{lemma}
\begin{proof}
By the covering relations given in Theorem~\ref{thm:interval}, we can see that every covering relation always either increases or decreases the second component by one, and thus $H_{\Gamma}$ gives us a column toggle order.
\end{proof}

This lemma and Corollary~\ref{cor:togcolprorow} yield the following.

\begin{corollary}
If a consistent restriction function $R$ has $R(p)$ a non-empty interval for all $p\in P$, then $\togpro_{H_{\Gamma}}$ on $J(\Gamma(P,R))$ is conjugate to $\row$.
\end{corollary}

A second application comes from Cartesian products.

\begin{definition}
We say that a Cartesian embedding of a ranked poset $P$ into an ordered pair of ranked posets $(P_1,P_2)$ is an order and rank preserving map from $P$ into the Cartesian product $P_1\times P_2$. If $P$ has a Cartesian embedding into $(P_1,P_2)$, we can denote a poset element $p\in P$ by its coordinates $(p_1,p_2)$ under this map, where $p_1\in P_1$, $p_2\in P_2$.
\end{definition}

\begin{remark}
The identity map on $P_1\times P_2$ is always a Cartesian embedding.
\end{remark}

\begin{remark}
\label{rmk:cartembed}
Note that we care not only about the isomorphism class of $P_1\times P_2$, but also the specific decomposition. For example, given $P=[a]\times [b] \times [c]$, we want to think of the inclusion map with $P_1=[a]\times [b]$ and $P_2=[c]$ as a distinct Cartesian embedding from either the inclusion map with $P_1=[a]$ and $P_2=[b]\times [c]$ or $P_1=[a]\times [b] \times [c]$ and $P_2$ consisting of a single element.
\end{remark}

\begin{lemma}
\label{lem:cartesian}
Let $P$ be a ranked poset with a Cartesian embedding into ranked posets $(P_1,P_2)$. Let $H$ map the element of $P$ embedded at coordinate $(p_1,p_2)$ to the difference of ranks $\mathrm{rk}_{P_1}(p_1) - \mathrm{rk}_{P_2}(p_2)$. Then $H$ defines a column toggle order.
\end{lemma}

\begin{proof}
Since $W$ is an order and rank preserving map, covering relations in $P$ get mapped to covering relations in $P_1\times P_2$. In the Cartesian product, we have $(x_1,y_1)\lessdot (x_2,y_2)$ if and only if either $x_1=y_2$ and $y_1\lessdot y_2$, or $x_1\lessdot x_2$ and $y_1=y_2$. In either case, $\mathrm{rk}_{P_1}(x_1) - \mathrm{rk}_{P_2}(y_1)=\mathrm{rk}_{P_1}(x_2) - \mathrm{rk}_{P_2}(y_2)\pm 1$, where $\rk_P$ indicates any rank function for the poset $P$.
\end{proof}

This lemma and Theorem~\ref{thm:newprorow} yield the following corollary.

\begin{corollary}
\label{cor:cartesian}
For $P$ and $H$ as in Lemma~\ref{lem:cartesian},
$\togpro_H$ on $J(P)$ is conjugate to $\row$ on $J(P)$.
\end{corollary}

\begin{remark}
Hyperplane promotion $\mathrm{Pro}_{\pi,v}$ of \cite{DPS2015} with respect to a lattice embedding $\pi$ can be thought of a special case of this. In particular, the $2^n$ choices of hyperplanes correspond to the $2^n$ ways that we can choose a subset $S\subseteq [n]$ and define a Cartesian embedding from $\mathbb{Z}^n$ to $\mathbb{Z}^{|S|}\times\mathbb{Z}^{n-|S|}$ by permuting coordinates so coordinates in $S$ go to one of the first $|S|$ copies of $\mathbb{Z}$, and coordinates not in $S$ get permuted to the last $n-|S|$ copies of $\mathbb{Z}$.

In particular, let $P$ be a poset with an $n$-dimensional lattice embedding $\pi$ and let $v \in\{-1,1\}^n$. Say that $i_1 < \ldots < i_a $ are the coordinates with $v_i=+1$ and $j_1 < \ldots < j_{n-a}$ are the coordinates with $v_j=-1$. Then the composition of $\pi$ sending $P$ to $\mathbb{Z}^n$ and the map sending $(x_1,\ldots,x_n)$ to $((x_{i_1},\ldots, x_{i_a}),(x_{j_1},\ldots, x_{j_{n-a}}))$ is a Cartesian embedding from $P$ to $P_1\times P_2$ for $P_1=\mathbb{Z}^a$ and $P_2=\mathbb{Z}^{n-a}$. The standard rank function on ${\mathbb{Z}^m}$ is  the sum of the coordinates. So if we write $v=v^+-v^-$ as a difference of $(0,1)$-vectors, then $\langle \pi(p),v^+\rangle = \rk_{P_1}(\pi(p))$, $\langle \pi(p),v^-\rangle = \rk_{P_2}(\pi(p))$ and thus $\langle \pi(p),v\rangle = \rk_{P_1}(\pi(p))-\rk_{P_2}(\pi(p))$. Partitioning poset elements by which hyperplane  $\langle \pi(p),v\rangle=i$ they lie on   creates a column toggle order.

For example, the vectors $(1,1,-1)$, $(1,-1,-1)$, and $(1,1,1)$ correspond (respectively) to the three Cartesian embeddings in Remark \ref{rmk:cartembed}.
\end{remark}

\subsection{Toggle-promotion is Bender-Knuth promotion}
\label{pro=pro}
In this subsection, we prove Theorem~\ref{thm:prorow}, showing that the bijection between $\inc{R}{P}$ and order ideals of $\Gamma(P,R)$ discussed in Sections~\ref{subsec:gamma1} takes increasing labeling promotion $\incpro$ to toggle-promotion $\togpro_{H_{\Gamma}}$ on $J(\Gamma(P,R))$ whenever $H_{\Gamma}$ is a column toggle order. This gives a new realm in which promotion and rowmotion have the same orbit structure. 

We begin by showing that $\incpro$ on increasing labelings $\inc{R}{P}$ exactly coincides with toggle-promotion on order ideals of $\Gamma(P,R)$, since the toggles $T^k_{H_{\Gamma}}$ exactly coincide with the generalized Bender-Knuth involutions $\rho_k$.

\begin{theorem}
\label{thm:proequalspro}
$\inc{R}{P}$ under $\incpro$ is in equivariant bijection with $J(\Gamma(P,R))$ under $\togpro_{H_{\Gamma}}$.
\end{theorem}

This follows from the lemma below.

\begin{lemma}
The map from  $\inc{R}{P}$ to order ideals in $\Gamma(P,R)$ equivariantly takes the generalized Bender-Knuth involution $\rho_{k}$ to the toggle operator $T_{H_{\Gamma}}^k$.
\end{lemma}

\begin{proof}
Recall that the column toggle order $H_{\Gamma}$ maps $(p,k)\in\Gamma(P,R)$ to $k$, so $T_{H_{\Gamma}}^k$ will toggle all elements in $\Gamma(P,R)$ of the form $(p,k)$.

It is easier to think about how toggling a single $(p,k)$ in $\Gamma(P,R)$ affects the corresponding increasing labeling.
When we toggle an individual $(p,k)$ in an order ideal $I$, there are three cases.

In the first case, $(p,k)$ can be toggled out. It can only be toggled out if it is a maximal element of the order ideal, which means that the corresponding increasing labeling gives the label $k$ to $p$. When we toggle $(p,k)$ out of $I$, the corresponding increasing labeling will now give the label $R(p)_{>k}$ to $p$, and the result is an increasing labeling. This is exactly the effect of $\rho_k$  in the first case.

In the second case, $(p,k)$ can be toggled in. This either means that $(p,R(p)_{>k})$ is in $I$, or no $(p,k')$ is in $I$. In both cases, the corresponding increasing labeling starts with $p$ being labeled with $R(p)_{>k}$ and getting reduced to $k$. This is exactly the effect of $\rho_k$  in the second case.

In the third case, $(p,k)$ can neither be toggled in nor out of $I$. This means 
that changing $p$ to $R(p)_{>k}$ (or vice versa) does not result in an increasing labeling.
\end{proof}

In essence, the generalized definition of $\rho_k$ was exactly constructed to coincide with $T_H^k$.

\begin{example}
\label{ex:proequalspro}
Consider Figure~\ref{fig:TogIsProEx_intro}, where we have a poset with a consistent restriction function $R$. On the left, we have an increasing labeling of this poset, and on the right, the corresponding order ideal in $\Gamma(P,R)$.

If we were to perform $\rho_2$ on the increasing labeling, we would try to increase the label of $b$ from 2 to 3, and we can, so we do. This corresponds to toggling $(b,2)$ and it being removed from the order ideal. We would also try to decrease the label of $c$ from 4 to 2, and we can, so we do. This is exactly the result of toggling $(c,2)$ and it being added to the order ideal.

Similarly, consider $\rho_4$. We would try to increment the label of $c$ from 4 to 5, but we cannot. The fact that the result is not an increasing labeling is equivalent to the fact that we cannot toggle out $(c,4)$ due to the presence of $(d,5)$ in the order ideal. Similarly, we would try and decrement the label of $d$ from 5 to 4, but we cannot. This corresponds to not being able to toggle in $(d,4)$. While $R(e)$ contains 4, the label it has is not currently 4, nor would decreasing it to the next lowest available label make it 4, so $\rho_4$ does nothing. This is equivalent to $(e,4)$ not being near enough to the boundary of the order ideal for it to have a chance at being toggled in or out.
\end{example}

Now, we may finally prove our main result of this section.

\begin{proof}[Proof of Theorem~\ref{thm:prorow}]
By Theorem~\ref{thm:proequalspro}, we know that the bijection between $\inc{R}{P}$ and $\Gamma(P,R)$ carries $\incpro$ on $\inc{R}{P}$ to $\togpro_{H_{\Gamma}}$ on $J(\Gamma(P,R))$. Then by Corollary~\ref{cor:togcolprorow}, if $H_{\Gamma}$ is a column toggle order, then $\togpro_{H_{\Gamma}}$ on $J(\Gamma(P,R))$ is conjugate to rowmotion.
Thus, the result follows.
\end{proof}

\begin{proof}[Proof of Corollary~\ref{cor:prorow}]
This follows, since by Lemma~\ref{lem:GammaColTog}, $H_{\Gamma}$ is a column toggle order for $\Gamma(P,q)$. 
\end{proof}

\begin{remark}
In this paper, we have shown the two components of the  proof both hold more generally; Theorem~\ref{thm:prorow} is a particular case where both components hold.  Theorem~\ref{thm:proequalspro} shows Bender-Knuth promotion on increasing labelings corresponds to toggle-promotion for a generic consistent restriction function $R$, not only the ones for which $H_{\Gamma}$ is a column toggle order. Similarly, Theorem~\ref{thm:newprorow} shows toggle-promotion is conjugate to rowmotion not only for $\Gamma(P,q)$, but for any poset which can be given a column toggle order.
\end{remark}

Finally, we obtain  as a corollary of Corollary~\ref{cor:prorow} and Theorem~\ref{thm:res} the following resonance result on order ideals in $\Gamma(P,q)$ under rowmotion.
Let $d$ be the toggle group element conjugating $\row$ to $\togpro_{H_{\Gamma}}$ (Corollary~\ref{cor:prorow} guarantees the existence of such an element).

\begin{corollary}
\label{cor:rowres}
Let $\varphi$ denote the map from an order ideal in $\Gamma(P,q)$ to the corresponding increasing labeling in $\inc{q}{P}$. Then
$(J(\Gamma(P,q)),\langle\row\rangle,\mbox{\rm Con}\circ\varphi\circ d)$ exhibits resonance with frequency~$q$. 
\end{corollary}

\section{An inverse map (of sorts)}
\label{sec:converse}
Thus far, we have shown that given a poset and a consistent restriction function, one can construct posets whose order ideals are in bijection with strictly or weakly increasing labelings of the original poset. One might consider when the reverse can be done. That is, given a poset, when can one construct an auxiliary poset and consistent restriction function so that order ideals in the original poset are in bijection with strictly or weakly increasing labelings of the auxiliary poset? 

One potential issue is that our forward map constructs $\Gamma(P,R)$ and $\Gamma_2(P,R)$ with respect to a specific embedding inside $P\times\mathbb{Z}$.  Given an arbitrary poset $Q$, we may be able to realize it inside $P\times\mathbb{Z}$ for multiple posets $P$.

In particular, every poset $P$ naturally embeds inside $P\times\{1\}$, and order ideals in $P$ are in bijection with weakly increasing labelings of $P$ with largest entry 2 (via the map that sends an order ideal to the increasing labeling where an element gets the label 1 if it is in the order ideal, and 2 if it is not in the order ideal).

Even if we are given an injective map $Q\mapsto P\times\mathbb{Z}$, there does not seem to be a simple way to see if $Q$ is isomorphic to some $\Gamma(P,R)$. The only method that seems to work is to let $R(p)^*$ be the preimage of $\{p\}\times\mathbb{Z}$, make an intelligent guess on what $\max(R(p))$ should be, and see if the covering relations for $\Gamma(P,R)$ given by Definition~\ref{def:GammaOne} (or Theorem~\ref{thm:interval} if $R$ consists only of intervals) match the covering relations of $Q$ under the injective map. In some cases, this method can work out nicely.

For example, consider the triangular poset with $n$ minimal elements, as in Figure~\ref{fig:staircase}, for $n=3$. We think of covering relations going down and to the right as being covering relations in a chain $P_n$ of length $n$, and covering relations going down and to the left as being covering relations in $\mathbb{Z}$. This gives us an embedding of the triangular poset into $P_n\times\mathbb{Z}$.

One can see that this will exactly correspond to $\Gamma(P_n,R)$, where $P_n$ is a chain $a_1<\ldots <a_n$, and $R(a_i)=[i,2i]$. Increasing labelings of a chain with these restrictions exactly correspond to increasing sequences of length $n$ with largest possible entry $2n$, one of many known combinatorial families enumerated by the Catalan numbers \cite{Catalan}.

As a result of Theorem~\ref{thm:proequalspro}, we can see that toggle-promotion $\togpro_{H_{\Gamma}}$ on the triangular poset (which with this labeling is toggling by columns right to left) is exactly the same as $\incpro$ on the corresponding increasing labelings.

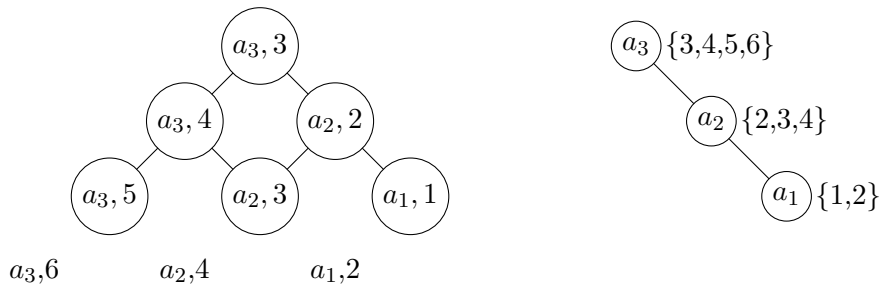
\begin{figure}[htbp]
\label{fig:staircase}
\begin{tikzpicture}
\begin{scope}[xshift=-5cm]
\coordinate (e) at (0,0);
\coordinate (b) at (1,1);
\coordinate (c) at (2,0);
\coordinate (d) at (-1,1);
\coordinate (f) at (-2,0);
\coordinate (a) at (0,2);
\coordinate (g) at (-3,-1);
\coordinate (h) at (-1,-1);
\coordinate (i) at (1,-1);
\draw[] (f) -- (d) -- (e) -- (b) -- (c) ;
\draw[] (d) -- (a) -- (b) ;
\node[fill=white,draw,circle,inner sep=.5ex] at (a) {$a_3,3$};
\node[fill=white,draw,circle,inner sep=.5ex] at (d) {$a_3,4$};
\node[fill=white,draw,circle,inner sep=.5ex] at (c) {$a_1,1$};
\node[fill=white,draw,circle,inner sep=.5ex] at (e) {$a_2,3$};
\node[fill=white,draw,circle,inner sep=.5ex] at (b) {$a_2,2$};
\node[fill=white,draw,circle,inner sep=.5ex] at (f) {$a_3,5$};
\node[inner sep=.5ex] at (g) {$a_3$,6};
\node[inner sep=.5ex] at (h) {$a_2$,4};
\node[inner sep=.5ex] at (i) {$a_1$,2};
\end{scope}

\begin{scope}
\coordinate (b) at (1,1);
\coordinate (a) at (2,0);
\coordinate (c) at (0,2);
\draw[] (a) -- (b) -- (c) ;
\node[fill=white,draw,circle,inner sep=.5ex] at (c) {$a_3$};
\node[fill=white,draw,circle,inner sep=.5ex] at (a) {$a_1$};
\node[fill=white,draw,circle,inner sep=.5ex] at (b) {$a_2$};
\node[right=.25cm] at (c) {\{3,4,5,6\}};
\node[right=.25cm] at (a) {\{1,2\}};
\node[right=.25cm] at (b) {\{2,3,4\}};
\end{scope}
\end{tikzpicture}
\caption{The triangular poset with the appropriate labeling to realize it as $\Gamma(P,R)$ for $P$ the three element chain and $R$ as given on the right.}
\end{figure}

\section*{Acknowledgments}
The authors would like to thank the anonymous referees for helpful comments. JS was supported by a grant from the Simons Foundation/SFARI (527204, JS).

\bibliographystyle{abbrv}
\bibliography{master}

\end{document}